\documentclass[letterpaper,10pt]{amsart}
\usepackage{amsmath,amssymb,amsxtra,amsfonts,amsthm,url,comment,microtype,enumerate}

\usepackage{a4wide}
\usepackage{tikz}
\usetikzlibrary{arrows,backgrounds,decorations,automata}
\usepackage{hyperref}

\theoremstyle{plain}
\newtheorem{theorem}{Theorem}[section]
\newtheorem{corollary}[theorem]{Corollary}
\newtheorem{proposition}[theorem]{Proposition}
\newtheorem{lemma}[theorem]{Lemma}
\theoremstyle{definition}
\newtheorem{definition}[theorem]{Definition}
\newtheorem{example}[theorem]{Example}

\newtheorem{remark}[theorem]{Remark}

\theoremstyle{remark}

\numberwithin{equation}{section}

\newcommand{\NN}{\mathbb{N}}
\newcommand{\ZZ}{\mathbb{Z}}
\newcommand{\CC}{\mathbb{C}}
\newcommand{\QQ}{\mathbb{Q}}
\newcommand{\W}{\mathcal{W}}
\newcommand{\Wa}{\tta\mathcal{W}}
\newcommand{\Wb}{\ttb\mathcal{W}}

\newcommand{\eps}{\varepsilon}

\newcommand{\tta}{{\tt a}}
\newcommand{\ttb}{{\tt b}}
\newcommand{\ttA}{{\tt A}}
\newcommand{\ttB}{{\tt B}}
\newcommand{\ttw}{{\tt w}}
\newcommand{\tts}{{\tt s}}
\newcommand{\ttf}{{\tt f}}
\newcommand{\ttl}{{\tt l}}

\makeatletter
\DeclareRobustCommand\bigop[2][1]{%
  \mathop{\vphantom{\sum}\mathpalette\bigop@{{#1}{#2}}}\slimits@
}
\newcommand{\bigop@}[2]{\bigop@@#1#2}
\newcommand{\bigop@@}[3]{%
  \vcenter{%
    \sbox\z@{$#1\sum$}%
    \hbox{\resizebox{\ifx#1\displaystyle#2\fi\dimexpr\ht\z@+\dp\z@}{!}{$\m@th#3$}}%
  }%
}
\makeatother

\DeclareMathOperator{\tr}{Tr}

\newcommand{\s}[1]{\left(\begin{smallmatrix} #1 \end{smallmatrix}\right)}   

\newcommand{\exend}{\hfill $\Diamond$}

\makeatletter
  \def\vhrulefill#1{\leavevmode\leaders\hrule\@height#1\hfill \kern\z@}
\makeatother

\newcounter{nootje}
\newcommand\noot[1]
 {\marginpar{\small\begin{minipage}{30mm}\begin{flushleft}\thenootje : #1\end{flushleft}\end{minipage}}\addtocounter{nootje}{1}}
\allowdisplaybreaks

\begin{document}

\title{Relative position in binary substitutions}


\author[M. Coons]{Michael Coons}
\address{Department of Mathematics and Statistics, California State University, Chico, California, USA}
\email{mjcoons@csuchico.edu}
\urladdr{https://mcoons-math.github.io}

\author[C. Ramsey]{Christopher Ramsey}
\address{Department of Mathematics and Statistics, MacEwan University,  Edmonton, Alberta, Canada}
\email{ramseyc5@macewan.ca}
\urladdr{https://sites.google.com/macewan.ca/chrisramsey/}

\author[N. Strungaru]{Nicolae Strungaru}
\address{Department of Mathematics and Statistics, MacEwan University, Edmonton, Alberta, Canada,
and
\newline \hspace*{\parindent}
Institute of Mathematics ``Simon Stoilow'', 
Bucharest, Romania}
\email{strungarun@macewan.ca}
\urladdr{https://sites.google.com/macewan.ca/nicolae-strungaru/home}

\date{\today}

\begin{abstract}
Given an infinite word $\ttw$ on a finite alphabet, an immediate  question arises:~can we understand the frequency of letters in $\ttw$\,? For words that are the fixed points of substitutions, the answer to this question is often `yes'---the details and methods of these answers have been well-documented. In this paper, toward a better-understanding of the fixed points of binary substitutions, we delve deeper by investigating, in fine detail, the position of letters by defining various position functions and proving results about their behavior. Our analysis reveals new information about the Fibonacci substitution and the extended Pisa family of substitutions, as well as a new characterization of the Thue--Morse sequence.
\end{abstract}

\maketitle

\section{Introduction}

Sequences over finite alphabets are pervasive in mathematics. Questions surrounding them have inspired the development of whole areas of mathematics---combinatorics on words, analytic number theory, symbolic dynamics, to name a few. For specific examples one need look no further than the prime number theorem and the Riemann hypothesis, both of which can be stated in terms of the Liouville $\lambda$-function---a binary sequence over the alphabet $\{-1,1\}$ which indicates the parity of the number of prime divisors of an integer. Such questions often concern the frequency of the values (e.g., prime number theorem) or the variance from that frequency (e.g., Riemann hypothesis). In this paper, our primary object of concern are substitution sequences, and, in particular, those acting on binary (two-letter) alphabets. By \emph{binary substitution}, we mean a map $\varrho$ from binary words to binary words such that it is a homomorphism---the natural operation on words being concatenation. Being a homomorphism, a substitution is naturally defined by how it acts on the binary alphabet. A ubiquitous example \cite{AS1999} is the Thue--Morse substitution $\varrho_{\rm TM}$ which is defined by
\[
\varrho_{\rm TM}:
\begin{cases}
    \tta \to \tta\ttb \\
    \ttb \to \ttb\tta\,.
\end{cases}
\] The one-sided infinite word that is the unique fixed point of this substitution which starts with the seed $\tta$,
\[
\ttw_{\rm TM}=\lim_{n\to\infty}\varrho_{\rm TM}^n(\tta)=\tta\ttb\ttb\tta\ttb\tta\tta\ttb\ttb\tta\tta\ttb\tta\ttb\ttb\tta\ttb\tta\tta\ttb\tta\ttb\ttb\tta\tta\ttb\ttb\tta\ttb\tta\tta\ttb\cdots,
\]
is often viewed as an infinite sequence (the so-called \emph{Thue--Morse sequence} or \emph{Prouhet--Thue--Morse sequence}), and is a paradigmatic example in several areas---most notably, number theory, dynamical systems, and theoretical computer science.

A binary substitution is a robust object. Along with it, comes a substitution matrix, which is nonnegative, and from that, one can often obtain information about the frequency of the letters in a fixed point, and hence the questions of the frequency of values can often be easily answered. Also, the questions concerning more nuanced behavior (speed of convergence) of the frequency are known for large classes of substitution sequences, in particular, for those of constant length. An under-appreciated result of Allouche, Mend\`es~France, and Peyri\`ere on Dirichlet series associated to such substitutions can be applied give a wealth of information on sequences arising as the fixed points of constant length substitutions. For non-constant length substitutions, less is known, but a result of Bell \cite{B2008} ensures that the logarithmic frequency of words in general substitutions exists.

Here, toward a better-understanding of the fixed points of binary substitutions, we investigate, in fine detail, the position of letters by defining various position functions and proving results about their behavior. We accomplish this here as follows. In Section \ref{sec:prelim}, we define the (relative) position functions and focus on preliminary results concerning these functions and their interaction with various operators on words. In Section \ref{sec:meanv}, we study the relationship between the means of our position functions and the standard letter frequency. Section \ref{sect:fib} contains an extended study of the Fibonacci substitution and the extended Pisa family of substitutions; in particular, we give a characterization of the Fibonacci word in terms of its relative position function. In Section \ref{sec:linear}, we characterize words that give rise to exact and asymptotically linear relative position functions. We conclude this paper in Section \ref{sec:TM}, where we give obtain a new characterization of the Thue--Morse sequence---it is the only sequence on $\{\pm1\}$ that is equal to its own relative position sequence.

\section{Preliminaries}\label{sec:prelim}

In this paper, we will look (usually) at infinite one-sided binary words over a two letter alphabet, whose elements we call \emph{bits}. That is, considering the alphabet $\Sigma=\{\tta,\ttb\}$ having bits $\tta$ and $\ttb$, we look at words
\[
\ttw=\ell_0 \ell_1 \ell_2 \cdots \ell_n \cdots
\]
where $\ell_n\in\Sigma$. Denote the set of finite words over $\Sigma$ by $\Sigma^*$, and the set of infinite words over $\Sigma$ by $\Sigma^{\omega}$. Further, set $\Sigma^{\infty}=\Sigma^*\cup\Sigma^\omega$. Here, for any finite word $w\in\Sigma^*$, $w\Sigma^\omega$ denotes the set of subwords of $\Sigma^{\omega}$ beginning with the word $w$, that is, having $w$ as a \emph{prefix}. For example, $\Sigma^\omega=\tta\Sigma^\omega\cup\ttb\Sigma^\omega$. Concatenation of words will be written in the usually power notation; for example, $\tta^3=\tta\tta\tta$ and $(\tta\ttb)^2=\tta\ttb\tta\ttb$. We write $(\ell_0\ell_1\cdots\ell_n)^\omega$ to denote the infinite periodic word with periodic part $\ell_0\ell_1\cdots\ell_n$. The length of a word $\ttw$ is denoted by $|\ttw|$, and the number of $\tta$'s and $\ttb$'s occurring in $\ttw$ are denoted by $|\ttw|_\tta$ and $|\ttw|_\ttb$, respectively. So, $|\ttw|=|\ttw|_a+|\ttw|_\ttb$.

Throughout this paper, we assume that both letters $\tta$ and $\ttb$ appear infinitely many times in any infinite binary word ${\tt w}$, or equivalently that $\ttw$ is not eventually $1$-periodic. We separate out this special subset $\W\subset\Sigma^\omega$ with the notation
\[
\W :=\{ \ttw\in\Sigma^\omega : |\ttw|_\tta=\infty\ \mbox{and}\ |\ttw|_\ttb=\infty \} \,.
\]
We also use the analogous notation to the above to indicate words starting with a given prefix for this special subset; for example, $\W=\tta\W\cup\ttb\W$, where, for a finite word $w$, we write $w\W$ to indicate the subset of words of $\W$ having prefix $w$.

While the words we studying are, in general, binary words, most of the examples we focus on are coming from substitutions---for this reason we prefer to use $\{\tta,\ttb\}$ as the alphabet instead of $\{0,1\}$. This choice allows us use the standard notation for level-$n$ super-tiles; that is, if $\sigma$ is a substitution on $\Sigma$, we set $\ttA_n:=\sigma^n(\tta)$ and $\ttB_n:=\sigma^n(\ttb)$.

\smallskip
Let us now introduce the notions of factor words and isomorphic words. Let $\ttw=w_0w_1\cdots$ and $\ttl=\ell_0\ell_1\cdots$ be words on alphabets $\Sigma$ and $\Sigma'$, and assume that the alphabets are \textit{reduced}, meaning
\[
\Sigma=\{ w_n : n \in \ZZ_{\geqslant 0} \} \quad\mbox{and}\quad \Sigma'=\{ \ell_n : n \in \ZZ_{\geqslant 0} \} \,.
\]
We say that $\ttl$ is a \textit{factor} of $\ttw$ if there exists a mapping $\sigma:\Sigma \to \Sigma'$ such that, for all $n \in \ZZ_{\geqslant 0}$ we have
\[
\ell_n=\sigma(w_n) \ .
\]
We will often call the mapping $\sigma$ a \textit{coding}. We say that $\ttw$ and $\ttl$ are \emph{isomorphic} if there exists such a mapping $\sigma$ which is a bijection.
It is easy to see that $\ttw$ and $\ttl$ are isomorphic if and only if each is a factor of the other.

Finally, given a sequence $\{x_n\}_{n \in \NN}$ which only takes finitely many values, we will often abuse notation and think about it as being the word
\[
\ttw = x_1x_2\cdots x_n \cdots \,.
\]
Note here that this will introduce a shift in position, as $x_1$ is in position $0$.

\smallskip

With the above `stringology' firmly noted, we move on to the definition of functions that will play important roles in what follows. Here, and below, $\NN$ denotes the set of positive integers, and $\ZZ_{\geqslant 0}$ denotes the set of non-negative integers, and $\CC$ denotes the set of complex numbers.

\begin{definition} Let $\ttw \in \W$. We define the \emph{position functions} $p_{\tta, \ttw}(n)$ and $p_{\ttb,\ttw}(n)$ as the position of the $n$-th occurrence of $\tta$ and $n$-th occurrence of $\ttb$ in $\ttw$, respectively, and we define the \emph{relative position function} $r_\ttw(n)$ as
\[
r_\ttw(n):=p_{\ttb,\ttw}(n)-p_{\tta,\ttw}(n) \,.
\]
Provided the context is clear, we use $p_\tta, p_\ttb$, and $r$ in place of $p_{\tta, \ttw},p_{\ttb,\ttw}$, and $r_\ttw$, respectively.
\end{definition}

\begin{definition}
For a sequence $s : \NN \to \CC$ , the \emph{difference sequence} $\Delta s:\NN \to \CC$ of $s$ is defined by
\[
\Delta s (n) := s(n+1)-s(n) \,.
\]
\end{definition}

Let $\ttw \in \W$. Note the following immediate consequences of the above definitions. Firstly, we have $r(n) \neq 0$ for any $n$. Secondly, the functions $p_\tta, p_\ttb$ are strictly increasing. In particular, $\Delta p_\tta, \Delta p_\ttb$ are positive sequences. Thirdly, the difference sequence $\Delta p_a(n)$ equals one plus the number of $\ttb$'s between the $n$-th $\tta$ and the next one. This is sometimes called the \textbf{sequence of $\ttb$-runs}. Similarly, the difference sequence $\Delta p_\ttb(n)$ equals one plus the number of $\tta$'s between the $n$-th $\ttb$ and the next one. Finally, since both $\tta$ and $\ttb$ appear infinitely many times in $\ttw$, the sets
\begin{align*}
  A &=\{ p_\tta(n): n \in \NN \} \\
  B &=\{ p_\ttb(n): n \in \NN \} \,,
\end{align*}
partition $\ZZ_{\geqslant 0}$ into two infinite sets. Moreover, any such partition corresponds uniquely to the images of the position functions for a word $\ttw$. In particular, $\ttw$ can be recovered from $A$ (or $B$).

With the above properties in hand, we now discuss which functions can be position functions or relative position functions. The following lemma is an immediate consequence of the definitions---since the proof is straightforward, we omit it.

\begin{lemma} Let $p :\NN \to \ZZ_{\geqslant 0}$. The following hold.
\begin{itemize}
  \item[(a)] There exists some $\ttw \in \Wa$ such that $p_\tta=p$ if, and only if, $p$ is strictly increasing, $p(1)=0$, and $\Delta p >1$ infinitely often.
  \item[(b)] There exists some $\ttw \in \Wa$ such that $p_\ttb=p$ if, and only if, $p$ is strictly increasing, $p(1)>0$, and $\Delta p >1$ infinitely often. \qed
\end{itemize}
\end{lemma}

The question of which $r: \NN \to \ZZ$ can occur as a relative position function is a bit trickier. As noted above, there are some restrictions on relative position functions, which eliminate many possibilities. Here, we show that a relative position function $r$ uniquely identifies $\ttw$.

\begin{lemma}\label{lemma2} Let $\ttw,\ttw' \in \W$ with relative position functions $r_\ttw$ and $r_{\ttw'}$, respectively. Then $\ttw=\ttw'$ if, and only if, $r_\ttw=r_{\ttw'}$.
\end{lemma}

\begin{proof} Necessity is clear, so we need only prove sufficiency. Toward this, suppose that $r_\ttw=r_{\ttw'}$. We will show by induction that $p_{\tta,\ttw}(n)=p_{\tta,\ttw'}(n)$, from which the result follows.

We know that $r_{\ttw}(1) = r_{\ttw'}(1)$. If this value is positive, then the first $\tta$ must appear before the first $\ttb$ and hence
\[
p_{\tta,\ttw}(1)=0=p_{\tta,\ttw'}(1) \,.
\]
On another hand, if $r_{\ttw}(1) = r_{\ttw'}(1) <0$ then
\begin{align*}
 p_{\ttb,\ttw}(1)&=0=p_{\ttb,\ttw'}(1)    \\
 p_{\tta,\ttw}(1)&=0-r_{\ttw}(1)=0-r_{\ttw'}(1)=p_{\tta,\ttw'}(1)  \,.
\end{align*}
This shows the claim for $n=1$.

 Now, suppose that $p_{\tta,\ttw}(k)=p_{\tta,\ttw'}(k)$ for all positive integers $k\leqslant n$. Then, by the definition of $p_{\tta,\ttw}$,  $p_{\tta,\ttw'}$, $r_\ttw$, and $r_{\ttw'}$, the first $n$  $\tta$'s in $\ttw$ appear at the positions
\[
A_n:= \{p_{\tta,\ttw}(1), p_{\tta,\ttw}(2), \ldots, p_{\tta,\ttw}(n) \},
\]
and that the first $n$ $\ttb$'s in $\ttw$ appear at the positions
\[
B_n:= \{  p_{\tta,\ttw}(1)+r_\ttw(1), p_{\tta,\ttw}(2)+r_\ttw(2), \ldots, p_{\tta,\ttw}(n)+r_\ttw(n) \} \,.
\]
Since $r_\ttw=r_{\ttw'}$, and $p_{\tta,\ttw}(k)=p_{\tta,\ttw'}(k)$ for $1 \leqslant k \leqslant n$, we also get that the first $n$  $\tta$'s in $\ttw'$ appear at the positions $A_n$ and that the first $n$ $\ttb$'s in $\ttw'$ appear at the positions $B_n$. Set
\[
M:= \min \big(\NN \backslash (A_n \cup B_n)\big) \,.
\]
We consider the two possible cases separately.

Suppose $r_\ttw(n+1)>0$. Then, in $\ttw$, the $(n+1)$-th $\tta$ appears before the $(n+1)$-th $\ttb$. So, for $k \geqslant n+1$ the $k$-th $\ttb$ appears after the $n+1$-th $\tta$. Also, for $k \geqslant n+2$ the $k$-th $\tta$ appears after the $(n+1)$-th $\tta$. Since the $M$-th position contains either an $\tta$ or a $\ttb$, which is neither among the first $n$ $\tta$'s nor the first $n$ $\ttb$'s, it must contain the $(n+1)$-th $\tta$. Thus $p_{\tta,\ttw}(n+1)=M$. Repeating the argument for $\ttw'$ instead of $\ttw$, we get $p_{\tta,\ttw'}(n+1)=M$, and so $p_{\tta,\ttw}(n+1)=M=p_{\tta,\ttw'}(n+1)$.

If instead, $r(n+1)<0$., a similar argument to the previous paragraph shows that $p_{\ttb,\ttw}(n+1)M=p_{\ttb,\ttw'}(n+1)$. Which, since $r_\ttw(n+1)=r_{\ttw'}(n+1)$, gives $p_{\tta,\ttw}(n+1)=p_{\tta,\ttw'}(n+1)$.
\end{proof}

The proof of Lemma \ref{lemma2} gives the following algorithm for reconstructing $\ttw$ from $r_\ttw$.

\begin{lemma}[Reconstruction algorithm] Let $r : \NN \to \ZZ$ be the relative position function of some word $\ttw \in \W$. Then, we can recover the word $\ttw$ from $r$ by the following simple algorithm.
\begin{itemize}
\item[\underline{Step 1.}] If $r(1) >0$, set $p_\tta(1)=0, p_\ttb(1)=r(1)$, otherwise set $p_\tta(1)=-r(1), p_\ttb(1)=0$.
\item[\underline{Step 2.}] For each $n \geqslant 2$ define
\begin{align*}
A_n&:= \{ p_\tta(1), p_\tta(2), \ldots, p_\tta(n) \}=A_{n-1} \cup \{ p_{\tta}(n) \} \\
B_n&:= \{  p_\ttb(1), p_\ttb(2), \ldots, p_\ttb(n) \} =B_{n-1} \cup \{ p_{\ttb}(n) \}.
\end{align*}
\item[\underline{Step 3.}] Set $k_{n+1}:= \min \big(\NN \backslash (A_n \cup B_n)\big).$ Then,
\begin{itemize}
  \item[(i)] if $r(n+1)>0$, set $p_\tta(n+1)=k_{n+1}$ and\, $p_{\ttb}(n+1)=k_{n+1}+r(n+1)$,
  \item[(ii)] if $r(n+1)<0$, set $p_\tta(n+1)=k_{n+1}-r(n+1)$ and\, $p_{\ttb}(n+1)=k_{n+1}$.
\end{itemize}
\item[\underline{Step 4.}] Increase $n$ by $1$ and go back to Step 2.\qed
\end{itemize}
\end{lemma}

\begin{remark} Given a relative position function $r$, in Step 3, the following two things must happen.
\begin{itemize}
\item[($\alpha$)] If $r(n+1)>0$, then we must have $k_{n+1} > p_{\tta}(n)$, $k_{n+1}+r(n+1)> p_{\ttb}(n)$, and $k_{n+1}+r(n+1) \notin A_n \cup B_n$.
\item[($\beta$)]  If $r(n+1)<0$, then we must have $k_{n+1} > p_{\ttb}(n)$, $k_{n+1}-r(n+1)> p_{\tta}(n)$, and $k_{n+1}-r(n+1) \notin A_n \cup B_n$.
\end{itemize}

\noindent Moreover, given any function $r: \NN \to \ZZ \backslash \{0 \}$ with $r(1)>0$, $r$ is the relative position function of some $\ttw$ if, and only if,
in the above algorithm the conditions $(\alpha)$ and $(\beta)$ always hold.\exend
\end{remark}

We now show that any increasing function $r$ with $r(1)>0$ is a relative position function, and we discuss the relation between monotonicity and the $\ttb$-runs. 

\begin{lemma}
For any increasing function $r: \NN \to \NN$, there exists some word $\ttw \in \Wa$ with $r=r_\ttw$.
\end{lemma}

\begin{proof} The idea of the proof is the same as the reconstruction algorithm---the key is that $r(n)>0$ for each $n$. We will determine the word $\ttw$ by the position of its letters. To this end, we start by defining $p_\tta(1)=0$ and $p_\ttb(1)=r(1)$, and proceed by induction.

For each $n \geqslant 1$, set
\begin{align*}
A_n&:= \{ p_\tta(1), p_\tta(2), \ldots, p_\tta(n) \}=A_{n-1} \cup \{ p_{\tta}(n) \} \\
B_n&:= \{  p_\ttb(1), p_\ttb(2), \ldots, p_\ttb(n) \} =B_{n-1} \cup \{ p_{\ttb}(n) \} \\
p_{\tta}(n+1)&:= \min \NN \backslash (A_n \cup B_n) \\
p_{\ttb}(n+1)&:=p_{\tta}(n+1)+r(n+1) \,.
\end{align*}
Since $(A_{n-1} \cup B_{n-1}) \subseteq (A_n \cup B_n)$, we immediately get that $p_{\tta}(n+1) > p_{\tta}(n)$. Next,
\[
p_{\ttb}(n+1)=p_{\tta}(n+1)+r(n+1) > p_{\tta}(n)+r(n) = p_{\tta}(n+1) \,.
\]
Moreover, by definition $p_{\tta}(n+1) \notin A_n \cup B_n$, which, with $p_{\ttb}(n+1)>p_\tta(n)$ and $p_{\ttb}(n+1)>p_\ttb(n)$, imply that $p_{\ttb}(n+1) \notin A_n \cup B_n$. This immediately implies that for all $n$ we have $A_n \cap B_n =\varnothing$. Since $A_1 \subsetneq A_2  \subseteq \cdots $ and $B_1 \subsetneq  B_2  \subsetneq \cdots $ are nested, we get that the unions
\[
A := \bigcup_n A_n \qquad\mbox{and}\qquad B := \bigcup_n B_n
\]
are disjoint infinite sets.

Finally, $p_\tta(1)=0$ and $p_\tta(n+1) > p_{\tta}(n)$ immediately imply that $p_\tta(n) \geqslant n-1$. Therefore
\[
\min \NN \backslash (A_n \cup B_n)= p_\tta(n+1)\geqslant n
\]
and hence
\[
\{1,2,3, \ldots, n \} \subseteq (A_{n+1} \cup B_{n+1})  \subseteq (A \cup B) \,.
\]
It follows that $\NN = A\cup B.$ Now, defining
\[
\ttw := \ell_0 \ell_1 \cdots \qquad\mbox{where}\qquad
\ell_k=
\begin{cases}
\tta & \mbox{ if } k \in A \\
\ttb &\mbox{ if } k \in B\, ,
\end{cases}
\]
since $A\cap B=\varnothing,$ we have that $r_\ttw=r$, which is the desired result.
\end{proof}

By applying an analogous argument, we get the following immediate corollary.

\begin{corollary} For any decreasing function $r : \NN \to \ZZ$ with $r(1)<0$, there exists some $\ttw \in \Wb$ with $r=r_\ttw$.\qed
\end{corollary}

\medskip

\begin{definition} The \emph{reflection operator}, $\bar{\cdot}$, on  $\Sigma^\infty$, is the morphism defined by $\overline{\tta}=\ttb$ and $\overline{\ttb}=\tta$.
\end{definition}

The following result is clear, so we omit the proof.

\begin{proposition}\label{prop12} The refection operator satisfies the following properties.
\begin{itemize}
    \item[(a)] $\overline{\Wa} = \Wb$ and $\overline{\Wb}=\Wa$.
    \item[(b)] Let $\ttw,\ttw'\in\W$. Then $\overline{\ttw}=\ttw'$ if, and only if, $r_{\ttw}=-r_{\ttw'}$\,.\qed
\end{itemize}
\end{proposition}

\noindent This proposition allows us to restrict our attention to $r_\ttw$ for $\ttw \in \Wa$ as needed. Note that the condition $\ttw \in \Wa$ is equivalent to $r(1)>0$, which is equivalent to $p_\tta(1)=0$. Also, $r(1)=k>1$ if, and only if, $\ttw \in\tta^{k-1}\ttb\W$. And, if $r(1)=k>1$, then $r(2),\ldots, r(k-1)\geqslant k$,

Note that the reflection operator induces an involution, $\tilde{\cdot}$, on the class of binary substitutions.

\begin{definition} Let $\sigma :\Sigma\to \Sigma^*$ be a binary substitution. Define $\tilde{\sigma}:\Sigma\to \Sigma^*$ by
\[
\tilde{\sigma} (\alpha)= \overline{ \sigma(\overline{\alpha})} \qquad \forall \alpha \in \Sigma = \{ \tta, \ttb \} \,.
\]
\end{definition}
A fast computation shows that $\tilde{\sigma} (\tts) =  \overline{ \sigma(\overline{\tts})}$ for all $\tts \in \Sigma^*$, and hence $\tilde{\sigma} (\ttw) =  \overline{ \sigma(\overline{\ttw})}$ for all $\ttw \in \W$.

We have the following characterization of the $\tta$ and $\ttb$ runs.

\begin{lemma} \label{lem:1} Let $\ttw \in \Wa$ and $k \geqslant 1$.
\begin{itemize}
\item[(a)] The following are equivalent.
\begin{itemize}
\item[(i)]$\ttb^k$ is a subword of $\ttw$.
\item[(ii)] $\sup \{ \Delta p_\tta(n) \} > k$.
\item[(iii)] $\Delta p_\ttb$ takes the value one at least $k-1$ times in a row.
\end{itemize}
\item[(b)]  The following are equivalent.
\begin{itemize}
\item[(i)]$\tta^k$ is a subword of $\ttw$.
\item[(ii)] $\Delta p_\tta$ takes the value one at least $k-1$ times in a row.
\end{itemize}
\end{itemize}
\end{lemma}

\begin{proof}
\textbf{(a)}
(i)$\Rightarrow$(ii). Since $\ttw$ starts with $\tta$, the word $\ttb^k$ appears somewhere after the first $\tta$. Let $j$ be so that the
$j$-th $\tta$ appears before $\ttb^k$ and $(j+1)$-th $\tta$ appears after $\ttb^k$. Then, there are at least $k$ bits between the $k$-th and $k+1$-th $\tta$, and hence
\[
p_{\tta}(j+1)-p_{\tta}(j) \geqslant k+1 \,.
\]

(ii)$\Rightarrow$(i). By (ii), there exists some $j$ such that
\[
p_{\tta}(j+1)-p_{\tta}(j) \geqslant k+1 \,.
\]
Then, there are $k$ bits between the $j$-th $\tta$ and the $j+1$-th $\tta$. Since these are consecutive $\tta$, all these bits have to be $\ttb$.

(i)$\Leftrightarrow$(iii) is clear.

\textbf{(b)} Is proven exactly as (a). The only difference is that, since $\tta^k$ may occur at the beginning of $\ttw$, and only at the beginning, we cannot conclude anything about $\Delta p_\ttb$.
\end{proof}

Note that Lemma~\ref{lem:1} implies that $\sup \{ \Delta p_\tta(n) \} - 1$ is the longest run of $\ttb$'s in $\ttw$ and that
$\sup \{ \Delta p_\ttb(n) \} - 1$ is the longest run of $\tta$'s appearing in $\ttw$. Combining this with the fact that $\Delta p_\tta \geqslant 1$, we have the following result.

\begin{lemma} \label{lemma1} Let $\ttw \in \Wa$.
\begin{itemize}
\item[(a)] If the function $r$ is strictly increasing, then ${\tt bb}$ is not a subword of ${\tt w}$.
\item[(b)] If\, ${\tt bb}$ is not a subword of ${\tt w}$, then $r$ is increasing.
\item[(c)] If\, $\ttb^k$ is a subword of $w$, there exists some $n$ so that $\Delta r(n+j) \leqslant 0$ for all $j\in\{1,\ldots,k-1\}$.
\item[(d)] If $\tta^k$ is a subword of $w$, there exists some $n$ so that $\Delta r(n+j) \geqslant 0$ for all $j\in\{1,\ldots,k-1\}$.
\end{itemize}
\end{lemma}
\begin{proof}
\textbf{(a)} Toward a contradiction, assume that $\ttb \ttb$ is a subword of $\ttw$. Then, by Lemma~\ref{lem:1}(a), there exists some $n$ such that $\Delta p_\ttb(n)=1$. Since $\Delta p_{\tta}(n) \geqslant 1$, we get that $\Delta r(n) \leqslant 0$. But, this contradicts the fact that $r$ is strictly increasing.

\textbf{(b)} Since $\ttb \ttb$ is not a subword of $\ttw$, by Lemma~\ref{lem:1}(a), we have $\Delta p_\tta(n) \leqslant 2$ for all $n$ and that $\Delta p_\ttb$ never takes the value $1$, meaning that $\Delta p_\ttb(n) \geqslant 2$ for all $n$. It follows immediately that $\Delta r \geqslant 0$ so that $r$ is increasing.

The proofs of \textbf{(c)} and \textbf{(d)} follow \emph{mutatis mutandis} of the proof of part \textbf{(a)} above.
\end{proof}

Note that, in Lemma~\ref{lemma1}(a), we cannot assume that $r$ is increasing, and, in Lemma~\ref{lemma1}(b), we cannot show that $r$ is strictly increasing. Indeed, here are two witnessing examples.
\begin{itemize}
  \item The word
\[
\ttw =\tta\ttb\tta\tta\ttb\ttb\tta\tta\tta\ttb \ttb\ttb\tta\tta\tta\tta\ttb\ttb\ttb\ttb\cdots =\tta \ttb \tta^2 \ttb^2 \tta^3 \ttb^3 \tta^4 \ttb^4 \cdots
\]
has the following sequence as relative position function
\[
1,2,2,3,3,3,4,4,4,4,\ldots \,.
\]
Here, $r$ is increasing, but $\ttw$ contains $\ttb \ttb$ as a subword.

Later, we will cover a more interesting example---the word obtained by adding the prefix $\tta \tta \ttb \ttb$ to the Fibonacci substitution $\sigma_2$ contains $\ttb \ttb$ and its relative position function $r$ satisfies $r(1)=r(2)=r(3)=2$ and is strictly increasing starting at $n=3$.
  \item The periodic word $\ttw = (\tta \ttb)^\omega$ does not contain $\ttb \ttb$ and has $r(n)=1$ for all $n$.
\end{itemize}

Before moving on to periodic sequences, we note that Lemma~\ref{lem:1} gives the following results on the possible boundedness of $\Delta r$.

\begin{lemma} Let $\ttw \in \Wa$. If $\ttb^k$ does not appear in $\ttw$, then $\Delta r$ is bounded if, and only if, $\Delta p_\ttb(n)$ is bounded.
\end{lemma}
\begin{proof} This following quickly noting that $\Delta r =\Delta p_\ttb -\Delta p_\tta$, and by Lemma~\ref{lem:1}, $\Delta p_\tta <k+1$.
\end{proof}

For a periodic word $\ttw$, all of the relative position functions are also periodic, as we shall prove here. But this relationship cannot be inverted---later, we will see some examples where $\Delta r_\ttw$ is periodic for an aperiodic word $\ttw$.

\begin{lemma}\label{lem:per-seq} Let $\ttw \in \Wa$. Then, the following are equivalent.
\begin{itemize}
  \item[(i)] $\ttw$ is periodic.
  \item[(ii)] $\Delta p_{\tta}$ is periodic.
  \item[(iii)] $\Delta p_{\tta}$ and $\Delta p_{\ttb}$ are periodic.
\end{itemize}
Moreover, if $\ttw =  w^\omega$ and $w$ contains $k\ \tta$'s and $j\ \ttb'$s, then $\Delta p_{\tta}$ is $k$-periodic and $\Delta p_{\ttb}$ is $j$-periodic.
\end{lemma}

\begin{proof}
(i)$\Rightarrow$(iii).
Let $\ttw = (w_0w_1\cdots w_{k+j-1})^\omega$, and let $w$ contain $k\ \tta$'s and $j\ \ttb'$s. Let $0 = l_0 < l_1 <l_2<\cdots < l_{k-1} \leqslant k+j-1$ be all the positions of $\tta$ in the word $w_0w_1\cdots w_{k+j-1}$. Then, the values of $p_a(1), p_a(2), p_a(3),\ldots$ are precisely
\[
l_0, l_1, l_2, \ldots, l_{j-1}, l_0+k+j, l_1+k+j, \ldots, l_{j-1}+k+j, l_0+2(k+j), l_1+2(k+j), \ldots, l_{j-1}+2(k+j), \ldots \,.
\]
Explicitly, if $n=qk+r$ where $1 \leqslant r \leqslant k$ when divided by $k$ then
\[
p_a(n)=q(k+j)+l_{r-1} \,.
\]
In particular, $\Delta p_a$ is the $k$-periodic sequence obtained by repeating
\[
l_1 -l_0, l_2-l_1, \ldots, l_{k-1}-l_{k-2}, l_0-l_{k-1}+k+j \,.
\]
The proof for $\Delta p_b$ is identical.

The claim that (iii)$\Rightarrow$(ii) is clear.

(ii)$\Rightarrow$(i).
Assume that $\Delta p_{\tta}$ is periodic with period $k$. Note that $0 = p_\tta(1) <  p_\tta(2) < \cdots < p_\tta(k)$. Set $m=p_\tta(k+1)-p_\tta(1)$. Then, for all $n=kq+t$ with $1 \leqslant t \leqslant k$ we have
\[
p_{\tta}(n)=mq+p_{\tta}(t) \,.
\]
This shows that the set of positions of $\tta$ in $\ttw$ is the finite union of the infinite arithmetic progressions
\[
p_\tta(1)+m \NN, p_\tta(2)+m \NN, \ldots, p_\tta(k)+m \NN \,,
\]
which is certainly periodic. The claim follows immediately.
\end{proof}

\begin{corollary} If $\ttw$ is periodic, so is $\Delta r(n)$. Moreover, if $\ttw=(w)^\omega$ and $w$ contains $k$ $\tt a$'s and $j$ $\ttb$'s, then ${\rm lcm}(k,j)$ is a period of $\Delta r$.\qed
\end{corollary}

To illustrate the above properties of periodic words $\ttw$, we summarize some specific examples with small period in the following table.

\begin{table}[htp]
\caption{Relative position sequences for some periodic words having small period.}
\begin{center}
\begin{tabular}{c|c|c|c}
$\ttw$ & $\Delta p_a$ & $\Delta p_b$ & $\Delta r$ \\ \hline
$(\tta\ttb)^\omega$ & $2,2,2,2,2,2,\ldots$ & $2,2,2,2,2,2,\ldots$ & $0,0,0,0,0,0,\ldots$\\
$(\tta\tta\ttb)^\omega$ & $1,2,1,2,1,2,\ldots$ & $3,3,3,3,3,3,\ldots$ & $2,1,2,1,2,1,\ldots$ \\
$(\tta\ttb\ttb\tta)^\omega$ & $3,1,3,1,3,1,\ldots$ & $1,3,1,3,1,3,\ldots$ & $-2,2,-2,2,-2,2,\ldots$ \\
$(\tta\tta\ttb\ttb)^\omega$ & $1,3,1,3,1,3,\ldots$ & $1,3,1,3,1,3,\ldots$ & $0,0,0,0,0,0,\ldots$
\end{tabular}
\end{center}
\label{default}
\end{table}%

\noindent We note that the example $\ttw=(\tta\tta\ttb\ttb)^\omega$ shows that the period of $\Delta r$ can be strictly smaller than the periods of $\Delta p_a$ and $\Delta p_b$.

In the remainder of this section, we discuss three more operators on the set of words $\W$---the deletion operator, the prefix operator, and the composition operator.

\begin{definition} Let $D_\tta: \W \to \W$ and $D_\ttb : \W \to \W$ be the operators which delete the first $\tta$ and first $\ttb$, respectively, in a word. The \emph{deletion operator} $D :\W \to \W$ is defined by
\[
D:= D_\tta \circ D_\ttb \,.
\]
\end{definition}

It is quite clear that $D_\ttb$ and $D_\tta$ commute with each other, so also, $D=D_\ttb \circ D_\tta$. Also, the reflection operator, $\bar{\cdot}$, satisfies the equalities $D_\tta\circ\bar{\cdot}=\bar{\cdot}\circ D_\ttb$ and $D_\ttb\circ\bar{\cdot}=\bar{\cdot}\circ D_\tta$, so that $D$ and $\bar{\cdot}$ commute, that is, $D\circ\bar{\cdot}=\bar{\cdot}\circ D$.

Let us now prove the following lemma.

\begin{lemma}\label{lem:Deletepositions} Let $\ttw \in \W$, $k:= \min \{ j : p_\ttb(j) > p_\tta(1) \},$ and $\ttw'=D_\tta(\ttw)$.
Then, for all $n$,
\[
p_{\tta,\ttw'}(n)=p_{\tta,\ttw}(n+1)-1 \,,
\qquad \mbox{and} \qquad p_{\ttb,\ttw'}(n)=
\begin{cases}
\, p_{\ttb,\ttw}(n)& \mbox{ if } n < k \\
\, p_{\ttb,\ttw}(n)-1& \mbox{ if } n \geqslant k \\
\end{cases}
\]
\end{lemma}

\begin{proof}
For the first equality, the $n$-th $\tta$ in $\ttw'$ comes from the $(n+1)$-th $\tta$ in $\ttw$. Since we deleted the first $\tta$ in the list, all positions of these $\tta$'s shift to the left one position.

For the second equality, note that the $\ttb$'s before the first $\tta$ stay the same, while the $\ttb$'s after the first $\tta$ in $\ttw$ are shifted to the left one position.
 \end{proof}

Lemma \ref{lem:Deletepositions} has the following consequences.

\begin{corollary}\label{cor2} Let $\ttw \in \W$ and $k >0$. Then $r_{D^k(\ttw)}(n)=r_\ttw(n+k)$ for all sufficiently large $n$.\qed
\end{corollary}

\begin{proposition}\label{p1} Let $\ttw \in \tta\ttb\W\cup \ttb\tta\W$. For all $n$, we have $r_{D(\ttw)}(n)=r_\ttw(n+1)$.\qed
\end{proposition}

With the language of the deletion operator in hand, we briefly turn back to periodic sequences.

\begin{lemma} Let $\ttw \in \W$. Then, $\Delta p_{\ttb}$ is periodic if, and only if, $D_{\tta}^{p_{\ttb}(1)}(\ttw)$ is periodic.
\end{lemma}

\begin{proof}
Since $\ttw' = D_{\tta}^{p_{\ttb}(1)}(\ttw)$ is just $\ttw$ with its initial run of $\tta$'s deleted, if there are any, then $\Delta p_{\ttb,\ttw} = \Delta p_{\ttb, \ttw'}$. Hence, because $\ttw' \in \mathcal W_\ttb$, the conclusion follows from Lemma \ref{lem:per-seq}.
\end{proof}

\begin{definition} The \emph{prefix operator}, ${\rm Pre}_{u}: \W \rightarrow u\W$ is the operator that adds the finite word $u$ to the start of any word. That is, ${\rm Pre}_{u}(\ttw) = u\ttw$.
\end{definition}

The identity ${\rm Pre}_{u}\circ \bar{\cdot} =\bar{\cdot} \circ {\rm Pre}_{\bar{u}}$ is immediate from the definitions. We record a few further properties of the prefix operator in the following lemmas.

\begin{lemma}\label{lem:12} We have $D \circ {\rm Pre}_{\tta \ttb} = D \circ {\rm Pre}_{\ttb \tta} =\mbox{Id}$. Moreover, for any $\ttw \in \W$, we have ${\rm Pre}_{\tta \ttb} \circ D(\ttw)=\ttw$ if, and only if, $\ttw \in \tta\ttb\W$. \qed
\end{lemma}

\begin{lemma} Let $\ttw \in \W$. Then
\[
p_{\tta,{\rm Pre}_\tta(\ttw)}(n)=
\begin{cases}
 0    & \mbox{ if } n=1 \\
 p_{\tta,\ttw}(n-1)+1   & \mbox{ if } n\geqslant 2,
\end{cases}
\qquad\mbox{and}\qquad
p_{\ttb,{\rm Pre}_\tta(\ttw)}(n) =  p_{\ttb,\ttw}(n)+1 \,.
\]
In particular, $r_{{\rm Pre}_\tta(\ttw)}(n) > r_{\ttw}(n)$.
\end{lemma}
\begin{proof}
Since ${\rm Pre}_\tta(\ttw)$ attaches the prefix $\tta$ for the word $\ttw$, the $n$-th $\ttb$ moves to the right one position, while the $n$-th $\tta$ moves one to the right and becomes the $(n+1)$-th $\tta$ in ${\rm Pre}_\tta(\ttw)$. The formulas for $p_{\tta,{\rm Pre}_\tta(\ttw)}$ and $p_{\ttb,{\rm Pre}_\tta(\ttw)}$ are now obvious.

Finally, we have
\[
r_{{\rm Pre}_\tta(\ttw)}(1)=p_{\ttb,{\rm Pre}_\tta(\ttw)}(1)-p_{\tta,{\rm Pre}_\tta(\ttw)}(1)=p_{\ttb,\ttw}(1)+1-0> p_{\ttb,\ttw}(1)-p_{\tta,\ttw}(1)= r_{\ttw}(1).
\]
And, for $n \geqslant 2$, we have
\begin{multline*}
r_{{\rm Pre}_\tta(\ttw)}(n)=p_{\ttb,{\rm Pre}_\tta(\ttw)}(n)-p_{\tta,{\rm Pre}_\tta(\ttw)}(n)=p_{\ttb,\ttw}(n)+1-(p_{\tta,\ttw}(n-1)+1)\\ = p_{\ttb,\ttw}(n)-p_{\tta,\ttw}(n-1)> p_{\ttb,\ttw}(n)-p_{\tta,\ttw}(n)= r_{\ttw}(n) \,.\qedhere
\end{multline*}
\end{proof}

By combining this result with Lemma~\ref{lem:12} and Proprosition \ref{prop12}, we get the following corollary.

\begin{corollary} Let $\ttw \in \W$. Then
\[
p_{\ttb,{\rm Pre}_\ttb(\ttw)}(n)=
\begin{cases}
 0    & \mbox{ if } n=1 \\
 p_{\ttb,\ttw}(n-1)+1   & \mbox{ if } n\geqslant 2,
\end{cases}
\qquad\mbox{and}\qquad
p_{\tta,{\rm Pre}_\ttb(\ttw)}(n) =  p_{\tta,\ttw}(n)+1 \,.
\]
In particular, $r_{{\rm Pre}_\ttb(\ttw)}(n) < r_{\ttw}(n)$. \qed
\end{corollary}

The deletion and prefix operators allow one to classify balanced words. Recall that a finite word $u$ is \emph{balanced} if it contains an equal number of $\tta$'s and $\ttb$'s.

\begin{lemma} Let $u\in\Sigma^*$ be a word of length $2k$. Then, $u$ is balanced if, and only if, $D^{k} \circ {\rm Pre}_u = \mbox{Id}$.
\end{lemma}

\begin{proof}
Let $u\in\Sigma^*$ be a word of length $2k$. Note that for any $\ttw \in \W$,
\[
D^{k} \circ {\rm Pre}_u(\ttw)= D^{k} (u \ttw )
\]
deletes the first $k$ $\tta$'s and the first $k$ $\ttb$'s in $u \ttw$.

Assume that $u$ is balanced. Since $u$ contains $k$ $\tta$'s and $k$ $\ttb$, $D^{k} (u \ttw )$ deletes the first $2k$ letters, which is the prefix $u$. That is, $D^{k} (u \ttw)=\ttw$.

Now suppose that $D^{k} \circ {\rm Pre}_u = \mbox{Id}$, and, toward a contradiction, let us further assume that $u$ is not balanced. Then, $u$ either contains at least $k+1$ $\tta'$s or it contains at least $k+1$ $\ttb'$s. Without loss of generality, suppose that $u$ contains at least $k+1$ $\tta'$s. Let $\ttw \in \W$ be arbitrary and consider
\[
D^{k} \circ {\rm Pre}_u (\ttw)= D^{k} (u \ttw) \,.
\]
Here, $u$ contains at most $k-1$ $\ttb$'s. Therefore, $D^k$ deletes all the $\ttb$'s in $u$, and deletes only $k$ of the at least $k+1$ $\tta$'s in $u$. This immediately implies that $D^{k} (u \ttw)$ starts with an $\tta$. Thus, we have $D^{k} \circ {\rm Pre}_u( \W) \subseteq \tta\W\subsetneq\W$, a contradiction that proves the result.
\end{proof}

We finish our focus on the prefix operator by showing that the addition of balanced words eventually shifts the (relative) position function(s).

\begin{lemma} Let $u\in\Sigma^*$ be a balanced word of length $2k$ and $\ttw \in \W$. Then, $p_{\tta,{\rm Pre}_{u}(\ttw)}(n+k) = p_{\tta,\ttw}(n)+2k$ and $p_{\ttb,{\rm Pre}_{u}(\ttw)}(n+k) = p_{\ttb,\ttw}(n)+2k$. In particular, $r_{{\rm Pre}_{u}(\ttw)}(n+k) = r_{\ttw}(n)$.
\end{lemma}

\begin{proof}
Each $\tta$ and $\ttb$ in $\ttw$ shifts to the right by $2k$ spots in ${\rm Pre}_{u}(\ttw)=u \ttw$, and has the rank (in order) increased by $k$, since there are $k$ bits of same type introduced before it. The claim follows.
\end{proof}

In the last part of this section, we discuss how the (relative) position functions behave on words under a binary substitution. Here, since we are interested in the set $\W$ of words that have infinitely many occurrences of $\tta$ and $\ttb$, we restrict ourselves to the consideration of binary substitutions $\varrho$ that satisfy $\tta,\ttb\in\varrho(\tta\ttb)$.

\begin{example} The reflection operator
\[
  \bar{\cdot}\, :\begin{cases} \tta \rightarrow \ttb \\
     \ttb \rightarrow \tta\end{cases}
\]
satisfies $\overline{ab}=ba$, so $a,b\in\overline{ab}$.\exend
\end{example}

In addition to the reflection operator, a particularly well-behaved family contains the cloning substitutions.

\begin{definition} Let $k >1$ be any positive integer. The \emph{cloning substitution} $\phi_k: \Sigma \to \Sigma^*$ is defined by
\[
\phi_k:\begin{cases} \tta\rightarrow \tta^k \\
\ttb\rightarrow \ttb^k \,.\end{cases}
\]
\end{definition}

Before proceeding, let us note in passing that for all $k \geqslant 2$ we have $\phi_k(\Wa) \subsetneq \Wa$ and $\phi_k(\Wb) \subsetneq \Wb$.

\begin{lemma}\label{lem23} Let $\ttw \in \W$ and $k \geqslant 2$.  Then, for all $m \in \ZZ_{\geqslant 0}$ and all $1 \leqslant j \leqslant k$, we have
    \begin{align*}
    p_{\tta,\phi_k(\ttw)}(mk+j)&=k\, p_{\tta,\ttw}(m+1)+j-1 \\
    p_{\ttb, \phi_k(\ttw)}(mk+j)&=k\, p_{\ttb,\ttw}(m+1)+j-1 \\
    r_{\phi_k(\ttw)}(mk+j) &= k\, r_{\ttw}(m+1) \,.
    \end{align*}
\end{lemma}
\begin{proof}
Consider the supertiles $\ttA=\phi_k(\tta)= \tta^k$ and $\ttB=\phi_k(\ttb)= \ttb^k$. Now, let $m \geqslant 0$. The $(m+1)$-th $\tta$ appears in $\ttw$ at position $p_\tta(m+1)$. This means that the $(m+1)$-th $\ttA$ is the $p_\tta(m+1)$ supertile in $\phi_k(\ttw)$. Since this is the $(m+1)$-th $\ttA$, and each $\ttA$ contains $k$ $\tta$'s and each $\ttB$ contains no $\tta$, there are exactly $mk$ $\tta's$ before this tile. Therefore, this $\ttA$ contains in order the $(mk+1)$-th, $(mk+2)$-th, $\ldots,$ $(mk+k)$-th $\tta$ in $\phi_k(\ttw)$. Recalling that the indices of $\ttw$ start at $0$, there are $p_\tta(m+1)$ supertiles before this $\ttA$, and each supertile contains $k$ letters. Therefore, the first position of this $\ttA$ is $k\,p_\tta(m+1)$. It follows that the $(mk+1)$-th, $(mk+2)$-th, $\ldots,$ $(mk+k)$-th $\tta$ in $\phi_k(\ttw)$ appear at positions $k\,p_\tta(m+1), k\,p_\tta(m+1)+1, k\,p_\tta(m+1)+2, \ldots, k\,p_\tta(m+1)+k-1$. This shows that for all $1 \leqslant j \leqslant k$, we have
\[
p_{\tta,\phi_k(\ttw)}(mk+j)= kp_{\tta,\ttw}(m+1) +j-1 \,.
\]

The result for $p_{\ttb,\phi_k(\ttw)}$ follows \emph{mutatis mutandis}, and the identity for $r_{\phi_k(\ttw)}$ follows directly.
\end{proof}

The formulas in Lemma~\ref{lem23} can be alternatively written in the form
    \begin{align*}
    p_{\tta,{\phi_k}(\ttw)}(n)&=k\, p_{\tta,\ttw}\left(\left\lfloor \frac{n}{k} \right\rfloor+1\right)+n-k\left\lfloor \frac{n}{k} \right\rfloor-1 \\
    p_{\ttb,{\phi_k}(\ttw)}(n)&=k\, p_{\ttb,\ttw}\left(\left\lfloor \frac{n}{k} \right\rfloor+1\right)+n-k\left\lfloor \frac{n}{k} \right\rfloor-1 \\
    r_{{\phi_k}(\ttw)}(n) &= k\, r_{\ttw}\left(\left\lfloor \frac{n}{k} \right\rfloor+1\right) \,.
    \end{align*}
where $\lfloor \cdot \rfloor$ is the standard floor function. In particular, we have that the ratios ${p_{\tta,{\phi_k}(\ttw)}(n)}/{n}$, ${p_{\ttb,{\phi_k}(\ttw)}(n)}/{n}$ and ${r_{{\phi_k}(\ttw)}(n)}/{n}$ have the same limits of indetermination as the ratios ${p_{\tta,\ttw}(n)}/{n}$, ${p_{\ttb,\ttw}(n)}/{n}$ and ${r_{\ttw}(n)}/{n}$, respectively, as $n\to\infty$.

\section{Mean values of (relative) position functions}\label{sec:meanv}

A nice consequence of Lemma~\ref{lem:1} is that for binary words where the number of consecutive identical letters is bounded, the relative position function is bounded, above and below, by linear functions.

\begin{lemma}\label{r-bounded-by-linear} Let $\ttw \in \Wa$. Denote the longest run of a single letter in ${\tt w}$ by $c$. If
\begin{align}
c:=& \sup\{ k :\ell_{n}=\ell_{n+1}=\cdots = \ell_{n+k-1},n \in \ZZ_{\geqslant 0} \} \nonumber\\
=& \sup \{ \Delta p_a(n) - 1, \Delta p_b(n) - 1 : n \in \ZZ_{\geqslant 0}\} \cup \{ p_b(1)\}<\infty \label{rel-den}\, ,
\end{align}
then,
\begin{align}
 p_\tta(n) & \leqslant (c+1)(n-1)  \label{eq1}\\
 p_\ttb(n) & \leqslant (c+1)n-1  \label{eq2}\,.
\end{align}
In particular, $(1-c)n+1 \leqslant r(n) \leqslant cn$.
\end{lemma}

\begin{proof} Since $\tta^{c+1}$ and $\ttb^{c+1}$ are not subwords of $\ttw$, by Lemma~\ref{lem:1}, we have both $\Delta p_\tta(n) \leqslant c+1$ and $\Delta p_\ttb(n) \leqslant c+1$. Since $p_\tta(1)=0$ and $p_\ttb(1)\leqslant c$, an easy induction yields
\eqref{eq1} and \eqref{eq2}.

Noting that $p_\tta(1)=0$, we have $p_\tta(n) \geqslant n-1$ and   $p_\ttb(n)\geqslant n$, so that in combinations of the upper bounds from the previous paragraph, give $c(1-n)+1 \leqslant r(n) \leqslant cn$, which finishes the proof of the lemma.
\end{proof}

\begin{remark} Note that the condition \eqref{rel-den} is equivalent to the sets of positions of $\tta$ and $\ttb$, respectively, being relatively dense in $\ZZ_{\geqslant 0}$. In particular, this
always holds for infinite words that are fixed points of primitive substitutions. For the definitions and examples related to relatively dense sets and substitutions, see Baake and Grimm \cite{TAO}.\exend
\end{remark}

In general, the upper bounds in \eqref{eq1} and \eqref{eq2} cannot be improved. To see this, note that for the periodic word $\ttw = (\tta \ttb^{c})^\omega$ we get equality in \eqref{eq1}, and for the periodic word $\ttw = (\tta^{c} \ttb)^\omega$ we get equality in \eqref{eq2}.

\begin{corollary} Let $\ttw \in \W$ whose longest letter run $c<\infty$. Then,
\begin{align*}
 p_\tta(n) & \leqslant (c+1)n-1  \\
 p_\ttb(n) & \leqslant (c+1)n-1 .
\end{align*} In particular, $-cn \leqslant r(n) \leqslant cn $.
\end{corollary}

\begin{proof} By applying the reflection operator, Lemma~\ref{r-bounded-by-linear} we have, for $\ttw \in \Wb$ that $p_\tta(n) \leqslant (c+1)n-1$,  $p_\ttb(n) \leqslant (c+1)(n-1)$, and $-cn \leqslant r(n) \leqslant c(n-1)-1$. The result follows immediately.
\end{proof}

For a word $\ttw$ and a letter $\alpha\in\Sigma$, denote by $\#_\alpha(n)=\#_{\alpha,\ttw}(n)$ the counting function of occurrences of $\alpha$ in the first $n$ bits of $\ttw$. Note that $\#_\alpha$ is an increasing function. For $\alpha \in \{ a, b \}$, it is immediate that $\#_\alpha(p_{\alpha}(n))=n,$ and $\#_{\alpha}(m)=n$ if, and only if, $p_\alpha(n) \leqslant m < p_{\alpha}(n+1)$ for all $n\geqslant 1$. In particular, $\#_\alpha \circ p_\alpha = \mbox{Id}$,
$p_\alpha \circ \#_\alpha \leqslant \mbox{Id},$ and $p_\alpha \circ (\#_\alpha+1) > \mbox{Id}$. Here, one has to be mindful of indices as $p_\alpha$ has domain $\ZZ_{\geqslant 0}$.

\begin{lemma}\label{dens} Let $\ttw \in \W$ and $\alpha\in\Sigma$. Then,
\vspace{.2cm}
\begin{itemize}
\item[(a)] $\displaystyle{\lim_{m\to\infty} \frac{\#_\alpha(m)}{m} = d \in (0,1]}$ if, and only if\, $\displaystyle{\lim_{n\to\infty}\frac{p_\alpha(n)}{n} = \frac{1}{d}}$,
\vspace{.2cm}
\item[(b)] $\displaystyle{\lim_{m\to\infty} \frac{\#_\alpha(m)}{m} = 0}$
if, and only if\, $\displaystyle{\lim_{n\to\infty}\frac{p_\alpha(n)}{n} = \infty}$.
\end{itemize}
\end{lemma}

\begin{proof}
We must show for $d \in [0,1]$ that $\lim_{m\to\infty} {\#_\alpha(m)}/{m} = d$ if, and only if, $\lim_{n}{n}/{p_\alpha(n)}=d$.

($\Rightarrow$). Let $\eps >0$, and suppose there exists some $M>0$ such that for any $m \geqslant M$,
\[
d-\eps < \frac{\#_\alpha(m)}{m} < d+\eps \,.
\]
Set $N:=\#_\alpha(M)+1$. For any $n\geqslant N$, $p_\alpha(n) \geqslant p_{\alpha}(N) = p_{\alpha}(c_\alpha(M)+1) > M$, so
\[
d-\eps < \frac{\#_\alpha(p_\alpha(n))}{p_\alpha(n)} < d+\eps,
\] which is precisely
\[
d-\eps < \frac{n}{p_\alpha(n)} < d+\eps.
\]
Therefore, $\lim_n {n}/{p_\alpha(n)} =d$, as desired.

($\Leftarrow$). Let $\eps >0$. Suppose there exists some $N>0$ such that for any $n \geqslant N$,
\[
d-\eps < \frac{n}{p_\alpha(n)} < d+\eps \,.
\]
Set $M:=p_\alpha(N)$. For any $m\geqslant M$, $n:= \#_\alpha(m) \geqslant \#_{\alpha}(M) =\#_{\alpha}(p_\alpha(N)) = N$. Since $p_\alpha\circ \#_\alpha \leqslant \mbox{Id}$,
\begin{align*}
\frac{\#_\alpha(m)} {m}  \leqslant \frac{\#_\alpha(m)} {p_\alpha(\#_\alpha(m))}
 = \frac{n}{p_\alpha(n)} \leqslant d+\eps
\end{align*}
And, as well,
\begin{align*}
    d-\eps < \frac{n+1}{p_\alpha(n+1)} = \frac{\#_\alpha(m) + 1}{p_\alpha(\#_\alpha(m) + 1)} < \frac{\#_\alpha(m) + 1}{m}.
\end{align*} Since $\lim_{n\to\infty}1/m=0$, the result follows.
\end{proof}

The following result follows from the proof of Lemma~\ref{dens}.

\begin{corollary}\label{cor:liminfsup}  Let $\ttw \in \W$ and $\alpha\in\Sigma$. Then
\begin{equation*}
\limsup_{m\to\infty} \frac{\#_\alpha(m)}{m} =\limsup_{n\to\infty} \frac{n}{p_\alpha(n)} \qquad\mbox{and}\qquad\liminf_{m\to\infty} \frac{\#_\alpha(m)}{m} =\liminf_{n\to\infty} \frac{n}{p_\alpha(n)} \,. \qed
\end{equation*}
\end{corollary}

\medskip

For the remainder of this section, we focus on results concerning the \emph{frequency} of the letters in a binary word $\ttw$. That is, for $\alpha\in\Sigma$, the limit $${\rm Freq}(\alpha)={\rm Freq}_\alpha(\ttw) := \lim_{m\to\infty} \frac{\#_\alpha(m)}{m}.$$ Our discussion will involve possible existence as well as consequences if the frequency exists value exists. Note that frequency is often called \emph{natural density}, especially in number theory. The following result, follows from Lemma \ref{dens}, and is also well-known---essentially folklore.

\begin{theorem}\label{t1} Let $\ttw \in \W$ and $d\in[0,1]$. The following are equivalent.
\begin{itemize}
  \item[(i)] ${\rm Freq}(\ttb)$ exists and is $d$.
  \item[(ii)] ${\rm Freq}(\tta)$ exists and is $1-d$.
  \item[(iii)] $\lim_{n\to\infty} {p_\ttb(n)}/{n}=\frac{1}{d}$.
  \item[(iv)] $\lim_{n\to\infty} {p_\tta(n)}/{n}=\frac{1}{1-d}$.\qed
\end{itemize}
\end{theorem}

\begin{corollary}\label{cor:asymp} Let $\ttw$ be any word such that ${\rm Freq}(\ttb)=d \in [0,1]$. Then,
\begin{equation*}
\lim_{n\to\infty} \frac{r(n)}{n}=\frac{1}{d}-\frac{1}{1-d}=\frac{1-2d}{d(1-d)} \,.\qed
\end{equation*}
\end{corollary}

\medskip

In particular, recalling that in a symbolic substitution the frequencies of the letters exist and are proportional to the right Perron-Frobenius eigenvector, the following result holds.

\begin{lemma}\label{sub:asympt} Let $\varrho : \Sigma \to \Sigma^*$ be a primitive substitution and $[u\ 1]^T$
be a right Perron--Frobenius eigenvector for the substitution matrix $M_\varrho$. Let $\ttw \in \W$ be a fixed point of $\varrho$; that is, $\varrho(\ttw)=\ttw$. Then, the following limits exist
\begin{align}
  {\rm Freq}(\tta) &=\frac{u}{u+1},\ {\rm Freq}(\ttb) =\frac{1}{u+1}\ \label{eqx}\\
  \lim_{n\to\infty} \frac{p_\tta(n)}{n}&= 1+\frac{1}{u},\ \
  \lim_{n\to\infty} \frac{p_\ttb(n)}{n}=u+1,\ \ \mbox{and}\ \
  \lim_{n\to\infty} \frac{r(n)}{n}=u-\frac{1}{u}=\frac{u^2-1}{u}. \nonumber
\end{align}
Moreover, all the above limits belong to $\QQ(\lambda_{PF})$, where $\lambda_{PF}$ is the Perron--Frobenius eigenvalue of the substitution.
\end{lemma}

\begin{proof}
The equations in \eqref{eqx} follow from the fact that the frequencies exist and are proportional to $[u\ 1]^T$. The existence and values of the remaining limits follow immediately from the previous results.

Let $M_\varrho \in \mathcal{M}_2(\ZZ)$ be the substitution matrix, $\lambda_{PF}$ its Perron--Frobenius eigenvalue. Then $0 \neq u \in \QQ(\lambda_{PF})$.
\end{proof}

\begin{proposition}
Let $\ttw \in \mathcal W$, and suppose that $\lim_{n\to\infty} {r(n)}/{n} =: r \in [-\infty, \infty]$; that is, $r$ takes a value in the extended real numbers. Then
\[
{\rm Freq}(\ttb) = \begin{cases}
0, & r = \infty \\
1, & r = -\infty \\
\frac{1}{2}, & r = 0 \\
\frac{2+r - \sqrt{4+r^2}}{2r}, & r\in (-\infty,0)\cup (0,\infty).
\end{cases}
\]
\end{proposition}
\begin{proof}
First, suppose $r=\infty$. Since ${r(n)}/{n} = {p_{\ttb}(n)}/{n} - {p_{\tta}(n)}/{n}$ is the difference of two positive sequences then we must have $\lim_{n\to\infty} \frac{p_{\ttb}(n)}{n} = \infty$. Thus, by Theorem \ref{t1}, ${\rm Freq}(\ttb) = 0$. The case of $r=-\infty$ follows similarly.

Suppose now, that $r\in (-\infty,\infty)$. Then, using Corollary \ref{cor:liminfsup},
\begin{multline*}
r = \limsup_n \frac{r(n)}{n}  = \limsup_{n\to\infty} \frac{p_{\ttb}(n)}{n} - \liminf_{n\to\infty} \frac{p_{\tta}(n)}{n} = \frac{1}{\displaystyle{\liminf_{m\to\infty} \frac{\#_{\ttb}(m)}{m}}} - \frac{1}{\displaystyle{\limsup_{m\to\infty} \frac{\#_{\tta}(m)}{m}}}
\\  = \frac{1}{\displaystyle{\liminf_{m\to\infty} \frac{\#_{\ttb}(m)}{m}}} - \frac{1}{\displaystyle{1 - \liminf_{m\to\infty} \frac{\#_{\ttb}(m)}{m}}} = \frac{\displaystyle{1 - 2\liminf_{m\to\infty} \frac{\#_{\ttb}(m)}{m}}}{\displaystyle{\liminf_{m\to\infty} \frac{\#_{\ttb}(m)}{m}\left(1 - \liminf_{m\to\infty} \frac{\#_{\ttb}(m)}{m}\right)}}\,.
\end{multline*}
Set $d_- := \liminf_m \frac{\#_{\ttb}(m)}{m}\in [0,1]$ to make the analysis cleaner. Hence,
\[
r=\frac{1-2d_-}{d_-(1-d_-)} \,.
\]
When $r=0$, we thus have $d_-=\frac{1}{2}$, while when $r \neq 0$, we obtain $rd_-^2-(2+r)d_-+1=0$ so that
\[
d_- =\frac{2+r \pm \sqrt{(2+r)^2-4r}}{2r}= \frac{2+r \pm \sqrt{4+r^2}}{2r} \,.
\]
Note that when $r \neq 0$, the quadratic function $f(x)=rx^2-(2+r)x+1$ satisfies $f(0)=1>0$ and $f(1)=-1<0$, and therefore, has exactly one root $d_- \in [0,1]$, which uniquely identifies $d_-$.

Repeating the same argument with $d_+ := \limsup_{m\to\infty} \frac{\#_{\ttb}(m)}{m} \in [0,1]$ yields
\[
r = \liminf_n \frac{r(n)}{n} = \frac{1-2d_+}{d_+(1-d_+)}\,.
\]
By uniqueness, $d_- = d_+$, so $d = {\rm Freq}(\ttb)$ exists. As noted above, if $r=0$, then $d=\frac{1}{2}$, and when $r\neq 0$, $d$ is equal to the single value satisfying
\[
d = \frac{2+r \pm \sqrt{4+r^2}}{2r}\in[0,1].
\]
If $r>0$, $\sqrt{4 + r^2} >r$ so that $\frac{2+r +\sqrt{4+r^2}}{2r}>1$, so that $d=\frac{2+r -\sqrt{4+r^2}}{2r}$. On the other hand, if $r<0$, then $-r = |r| < 2 + \sqrt{4 + r^2}$ and so
\[
\frac{2+r + \sqrt{4+r^2}}{2r} < 0
\]
Thus, again, $d=\frac{2+r -\sqrt{4+r^2}}{2r}$.
\end{proof}

\begin{theorem} If one of the limits
\[
r:= \lim_n \frac{r(n)}{n},\quad
p:= \lim_n \frac{p_\ttb(n)}{n},\quad\mbox{or}\quad
q:=\lim_n \frac{p_\tta(n)}{n}
\]
exist, then they all exist and $\frac{1}{p}+\frac{1}{q}=1$.
Moreover, on the extended reals,
\[
p= \frac{2+r + \sqrt{4+r^2}}{2} \quad \textrm{and} \quad q = \frac{2-r + \sqrt{4+r^2}}{2}
\]
\end{theorem}
\begin{proof}
The previous proposition proves the simultaneous existence of these limits.
Hence, we have $r=p-q$ and $\frac{1}{p}+\frac{1}{q}={\rm Freq}(\ttb) + {\rm Freq}(\tta)=1$. Thus, $r=p-\frac{p}{p-1}$, so that  $p^2-(2+r)p+r =0$, from which we obtain
\[
p= \frac{2+r \pm \sqrt{4+r^2}}{2}\,.
\]
By the previous proposition, when $r\in (-\infty,0)\cup(0,\infty),$
\[
p = \frac{1}{d} = \frac{2r}{2+r - \sqrt{4+r^2}} = \frac{2r(2+r + \sqrt{4+r^2})}{(2+r)^2 - (4+r^2)} = \frac{2+r + \sqrt{4+r^2}}{2}\,.
\]
Observe that this formula also holds true for $r\in\{0,\pm \infty\}$, and that $q = p-r = \frac{2-r + \sqrt{4+r^2}}{2}$.
\end{proof}

We finish this section by considering what happens to the asymptotics of the (relative) position functions under a binary substitution. This is addressed by the following result on the frequency of letters of a fixed point of a binary substitution.

\begin{theorem}
Let $\varrho : \Sigma \rightarrow \Sigma^*$ be a substitution with  substitution matrix $M_\varrho\in M_2(\mathbb Z)$ such that $\tta$ and $\ttb$ appear in $\varrho(\tta\ttb)$. If $\ttw \in \mathcal W$ is any word such that ${\rm Freq}_\tta(\ttw)$ exists, then
\[
\begin{bmatrix}
{\rm Freq}_\tta(\varrho(\ttw)) \\ {\rm Freq}_\ttb(\varrho(\ttw))
\end{bmatrix}
\ = \ \frac{1}{|\varrho(\tta)|\cdot{\rm Freq}_\tta(\ttw) + |\varrho(\ttb)|\cdot{\rm Freq}_\ttb(\ttw)}\cdot M_\varrho\cdot\begin{bmatrix}
{\rm Freq}_\tta( \ttw) \\ {\rm Freq}_\ttb( \ttw)
\end{bmatrix}\,.
\]
Moreover, if $M_\varrho$ is invertible and ${\rm Freq}_\tta(\varrho(\ttw))$ exists,
\[
\begin{bmatrix} {\rm Freq}_\tta(\ttw) \\ {\rm Freq}_\ttb(\ttw) \end{bmatrix} \ = \ \left\langle (M_\varrho^{-1})^*\begin{bmatrix} 1 \\ 1\end{bmatrix}, \begin{bmatrix} {\rm Freq}_\tta(\varrho(\ttw)) \\ {\rm Freq}_\ttb(\varrho(\ttw))\end{bmatrix}\right\rangle^{-1}\cdot M_\varrho^{-1} \cdot\begin{bmatrix} {\rm Freq}_\tta(\varrho(\ttw)) \\ {\rm Freq}_\ttb(\varrho(\ttw))\end{bmatrix}\,.
\]
\end{theorem}

\begin{proof}
Let $\ttw = \ell_0\ell_1\ell_2\cdots$. For each $m\geqslant 0$ define
\[
n_m : = \ |\rho(\ell_0\cdots \ell_m)|  \ = \ |\rho(\tta)|\cdot\#_{\tta,\ttw}(m) + |\rho(\ttb)|\cdot\#_{\ttb,\ttw}(m) \,.
\]
Hence,
\[
\lim_{m\rightarrow \infty} \frac{n_m}{m} \ = \ |\rho(\tta)|\cdot{\rm Freq}_\tta(\ttw) + |\rho(\ttb)|\cdot{\rm Freq}_\ttb(\ttw)\,.
\]
Now, using the entries of the substitution matrix $M=M_\varrho$, we get
\begin{align*}
\#_{\tta,\varrho(\ttw)}(n_m) & = M_{11}\cdot \#_{\tta,\ttw}(m) + M_{12} \cdot\#_{\ttb,\ttw}(m) \\
\#_{\ttb,\varrho(\ttw)}(n_m) & = M_{21}\cdot \#_{\tta,\ttw}(m) + M_{22}\cdot \#_{\ttb,\ttw}(m)\,.
\end{align*}
Thus,
\begin{align*}
    \lim_{m\rightarrow \infty} \frac{\#_{\tta,\varrho(\ttw)}(n_m)}{n_m} & = \left(\lim_{m\rightarrow \infty} \frac{m}{n_m} \right)\cdot\left(\lim_{m\rightarrow \infty} \frac{M_{11}\cdot \#_{\tta,\ttw}(m) + M_{12}\cdot \#_{\ttb,\ttw}(m)}{m}\right)
    \\ & = \frac{M_{11}\cdot {\rm Freq}_\tta(\ttw) + M_{12} \cdot {\rm Freq}_\ttb(\ttw)}{|\rho(\tta)|\cdot{\rm Freq}_\tta(\ttw) + |\rho(\ttb)|\cdot{\rm Freq}_\ttb(\ttw)}\,.
\end{align*}
Now, for any $n\geqslant n_0$, there exists $m\geqslant 0$ such that $n_m \leqslant n\leqslant n_{m+1}$. Since the counting function is increasing, this implies that
\[
\frac{n_m}{n_{m+1}}\cdot\frac{\#_{\tta,\varrho(\ttw)}(n_m)}{n_m} = \frac{\#_{\tta,\varrho(\ttw)}(n_m)}{n_{m+1}} \leqslant \frac{\#_{\tta,\varrho(\ttw)}(n)}{n} \leqslant \frac{\#_{\tta,\varrho(\ttw)}(n_{m+1})}{n_m} = \frac{n_{m+1}}{n_m}\cdot\frac{\#_{\tta,\varrho(\ttw)}(n_{m+1})}{n_{m+1}}\,.
\]
Finally, using the fact that
\[
\lim_{m\rightarrow \infty} \frac{n_m}{n_{m+1}} = \lim_{m\rightarrow\infty} \frac{n_m}{m}\cdot\frac{m+1}{n_{m+1}}\cdot\frac{m}{m+1} = 1,
\]
the squeeze theorem gives that
\[
{\rm Freq}_\tta(\varrho(\ttw)) = \frac{M_{11} \cdot{\rm Freq}_\tta(\ttw) + M_{12} \cdot{\rm Freq}_\ttb(\ttw)}{|\rho(\tta)|\cdot{\rm Freq}_\tta(\ttw) + |\rho(\ttb)|\cdot{\rm Freq}_\ttb(\ttw)}\,.
\]
The argument for ${\rm Freq}_\ttb(\varrho(\ttw))$ follows similarly.

Suppose now that the substitution matrix $M_\varrho$ is invertible and ${\rm Freq}_{\tta}(\varrho(\ttw))$ exists. Using the same setup as before, let $m\geqslant 0$ and $n_m = |\rho(\ell_0\cdots \ell_m)|$. Now, we turn things around to get
\begin{multline*}
m = \#_{\tta,\ttw}(m) + \#_{\tta,\ttw}(m)  = \left\langle\begin{bmatrix} 1 \\ 1\end{bmatrix}, \begin{bmatrix} \#_{\tta,\ttw}(m) \\ \#_{\ttb,\ttw}(m)\end{bmatrix}\right\rangle
\\ = \left\langle\begin{bmatrix} 1 \\ 1\end{bmatrix}, \ M_\varrho^{-1}\begin{bmatrix} \#_{\tta,\varrho(\ttw)}(n_m) \\ \#_{\ttb,\varrho(\ttw)}(n_m)\end{bmatrix}\right\rangle
 = \left\langle (M_\varrho^{-1})^*\begin{bmatrix} 1 \\ 1\end{bmatrix}, \begin{bmatrix} \#_{\tta,\varrho(\ttw)}(n_m) \\ \#_{\ttb,\varrho(\ttw)}(n_m)\end{bmatrix}\right\rangle\,.
\end{multline*}
Hence,
\[
\lim_{m\rightarrow \infty} \frac{m}{n_m} = \left\langle (M_\varrho^{-1})^*\begin{bmatrix} 1 \\ 1\end{bmatrix}, \begin{bmatrix} {\rm Freq}_\tta(\varrho(\ttw)) \\ {\rm Freq}_\ttb(\varrho(\ttw))\end{bmatrix}\right\rangle\,.
\]
Therefore, following in the same way as above
\begin{equation*}
\begin{bmatrix} {\rm Freq}_\tta(\ttw) \\ {\rm Freq}_\ttb(\ttw) \end{bmatrix} \ = \ \left\langle (M_\varrho^{-1})^*\begin{bmatrix} 1 \\ 1\end{bmatrix}, \begin{bmatrix} {\rm Freq}_\tta(\varrho(\ttw)) \\ {\rm Freq}_\ttb(\varrho(\ttw))\end{bmatrix}\right\rangle^{-1} M_\varrho^{-1} \begin{bmatrix} {\rm Freq}_\tta(\varrho(\ttw)) \\ {\rm Freq}_\ttb(\varrho(\ttw))\end{bmatrix}\,.\qedhere
\end{equation*}
\end{proof}

Of course, once you have the letter frequencies, then you have the relative position asymptotics.

\begin{example} Let $\ttw \in \W$, $k \geqslant 2$ and let ${\phi_k}$, the $k$-cloning substitution. The substitution matrix of the $k$-cloning composition operator is $kI_2$.
Using the above theorem, or Lemma \ref{lem23}, we get that
if one of the limits below exist, then all exist, and we have     ${\rm Freq}_\alpha(\ttw) =     {\rm Freq}_\alpha(\phi_k(\ttw)) $,  for $\alpha\in\Sigma$, and
\begin{align*}
\lim_{n\to\infty} \frac{r_{\ttw}(n)}{n} = &\lim_{n\to\infty} \frac{r_{\phi_k(\ttw)}(n)}{n},\quad
  \lim_{n\to\infty} \frac{p_{\tta,\ttw}(n)}{n}= \lim_{n\to\infty} \frac{p_{\tta,\phi_k(\ttw)}(n)}{n}\, ,\\
  &\mbox{and}\quad
    \lim_{n\to\infty} \frac{p_{\ttb,\ttw}(n)}{n}= \lim_{n\to\infty} \frac{p_{\ttb,\phi_k(\ttw)}(n)}{n}\, .
\end{align*}
\end{example}

\section{The Fibonacci and extended Pisa family of substitutions}\label{sect:fib}

In this section, we look at the extended Pisa family of substitutions, the canonical example of which, is the Fibonacci substitution. To introduce this family, we start this section by focusing on the Fibonacci substitution and emphasize some features which are exclusive to this example.

\begin{definition} We call the substitution,
\[
\varrho_F:\begin{cases}
\tta&\to \tta \ttb \\
\ttb&\to\tta \,,
\end{cases}
\]
the \emph{Fibonacci substitution}, and the resulting one-sided infinite fixed point,
\[
\ttf := \tta \ttb \tta \tta \ttb \tta \ttb \tta \cdots  = \lim_{n\to\infty} \varrho_F^n(\tta)\, ,
\]
the \emph{Fibonacci word}.
\end{definition}

We start with the following result.

\begin{theorem}\label{thm-fib}The Fibonacci word $\ttf$ has the following properties.
\begin{itemize}
  \item[(a)] $\Delta p_\tta$ is the sequence obtained from the Fibonacci word under the coding $(\tta,\ttb)=(1,2)$.
  \item[(b)] $\Delta p_\ttb$ is the sequence obtained from the Fibonacci word under the coding $(\tta,\ttb)=(2,3)$.
  \item[(c)] $r(n)=n$ and $\Delta r=1$. In particular, $\Delta r$ is periodic.
\end{itemize}
\end{theorem}

\begin{proof}
\textbf{(a)} Consider the level-$1$ supertiles $\ttA_1:=\varrho_F(\tta)=\tta \ttb$ and $\ttB_1 :=\varrho_F(\ttb)=\tta$. Each such supertile contains exactly one $\tta$, at the beginning. This means that, for all $n \geqslant 1$ the $n$-th $\tta$ is the first letter of the $n$-th level-$1$ supertile.

Now, if the $n$-th letter in the Fiboancci word is $\tta$, then the $n$-th supertile is $\ttA_1=\tta \ttb$. This means that the $n$-th $\tta$ is followed by $\ttb$, and then, since there are no two $\ttb$ in a row, followed by $\tta$. This implies that the distance between the $n$-th and $(n+1)$-th $\tta$ is $2$ whenever the $n$-th letter in the Fibonacci word is $\tta$.

Next, if the $n$-th letter in the Fiboancci word is $\ttb$, then the $n$-th supertile is $\ttB_1=\tta$. Since $\ttB_1\ttB_1$ never appears, this $\ttB_1$ is followed by $\ttA_1=\tta \ttb$. This implies that the distance between the $n$-th and $(n+1)$-th $\tta$ is $1$ when the $n$-th letter in the Fibonacci word is $\ttb$. Therefore
\[
\Delta p_\tta(n) =
\begin{cases}
  2 & \mbox{ if the } n \mbox{th letter in } \ttf \mbox{ is } \tta \\
  1 & \mbox{ if the } n \mbox{th letter in } \ttf \mbox{ is } \ttb.
\end{cases}
\]

\textbf{(b)} Consider the level-$2$ supertiles $\ttA_2 :=\varrho_F^2(\tta)=\tta \ttb \tta$ and $\ttB_2 :=\varrho_F^2(\ttb)=\tta \ttb$. Each such supertile contains exactly one $\ttb$, in the second position. This means that, for all $n \geqslant 1$ the $n$-th $\ttb$ is the second letter of the $n$-th level-$2$ supertile.

Now, if the $n$-th letter in the Fibonacci word is $\tta$, then the $n$-th level two supertile is $\ttA_2=\tta \ttb \tta$. The next supertile starts with $\tta \ttb$ since $\ttA_2=\tta \ttb \tta$ or $\ttB_2=\tta \ttb$. This implies that the distance between the $n$-th and $(n+1)$-th $\ttb$ is $3$ whenever the $n$-th letter in the Fibonacci word is $\tta$.

Next, if the $n$-th letter in the Fiboancci word is $\ttb$, then the $n$-th level-$2$ supertile is $\ttB_2=\tta \ttb$. The next level-$2$ supertile starts with $\tta \ttb$, again since both $\ttA_2$ and $\ttB_2$ start with $\tta \ttb$. This implies that the distance between the $n$-th and $(n+1)$-th $\ttb$ is $2$ whenever the $n$-th letter in the Fibonacci word is $\ttb$. Therefore
\[
\Delta p_\ttb(n) =
\begin{cases}
  3 & \mbox{ if the } n \mbox{th letter in } \ttf \mbox{ is } \tta \\
  2 & \mbox{ if the } n \mbox{th letter in } \ttf \mbox{ is } \ttb.
\end{cases}
\]

\textbf{(c)}
By (a) and (b) we have $\Delta r(n)=1$ for all $n$. By looking at the first $2$ letters in the Fibonacci word $\ttf$, we get that $r(1)=1$. This implies that the Fibonacci word is the only word such that $r(n)=n$.
\end{proof}

Let us note here in passing that Theorem~\ref{thm:6} will give a different proof of the fact that $r_{\ttf}(n)=n$. We prove below, in Theorem~\ref{fib-uni}, that the Fibonacci word is the unique non-periodic word which satisfies properties (a) and (b) in Theorem~\ref{thm-fib}, with $(1,2,2,3)$ replaced by any four numbers.

Let $\tau= \frac{1+\sqrt{5}}{2}$ be the golden ratio. By Lemma \ref{sub:asympt}, the Fibonacci $\ttf$ word satisfies
\[
{\rm Freq}_\ttf(\tta)=\frac{\tau}{\tau+1} \qquad\mbox{and}\qquad {\rm Freq}_\ttf(\ttb)=\frac{1}{\tau+1} \,,
\]
and
\[
\lim_{n\to\infty} \frac{r_\ttf(n)}{n}= \tau+1- \frac{\tau+1}{\tau}= \frac{\tau^2-1}{\tau}=1 \,.
\]
Of course this holds by the previous theorem but it is important to realize that the substitution itself was telling us that $r_{\ttf}(n)$ had to at least be asymptotically $n$.

We can characterize all substitutions for which $\lim_{n\to\infty} {r(n)}/{n}=1$. We require the following preliminary result.

\begin{lemma}\label{lem2} Let $M \in M_2(\ZZ)$. Then,  $[\tau\ 1]^T$ is a right eigenvector for $M$ if, and only if, there exists $m, n \in \ZZ$ such that
\[
M= \begin{bmatrix} m+n & m \\ m & n\end{bmatrix} = m\begin{bmatrix}
 1 & 1 \\
 1 & 0
\end{bmatrix} + n \begin{bmatrix}
    1 & 0 \\
    0& 1
\end{bmatrix}.
\]
Moreover, in this case, the eigenvalues of $M$ are $n+m \tau$ and $n+m \tau'$, where $\tau'=(1-\sqrt{5})/2$ is the algebraic conjugate of $\tau$.
\end{lemma}

\begin{proof}
($\Leftarrow$). This direction follows easily from
\[
\begin{bmatrix}m+n & m \\ m & n \end{bmatrix}\begin{bmatrix} \tau  \\ 1 \end{bmatrix}= \begin{bmatrix} m \tau + n \tau +m \\  m \tau + n \end{bmatrix} = (m \tau + n) \begin{bmatrix}
\tau  \\ 1\end{bmatrix}.
\]

($\Rightarrow$). Denote the entries of $M$ as
\[
M =\begin{bmatrix} k & l \\  m & n \end{bmatrix} \in M_2(\ZZ).
\]
The eigenvalue-eigenvector equation $ M\,[\tau\ 1]^T=\lambda\,[\tau\ 1]^T$ is equivalent to the linear system
\begin{align*}
  k \tau + l &= \tau \lambda \\
  m\tau + n &= \lambda \,.
\end{align*}
Therefore, $\lambda=\lambda_1 = m \tau +n = k-l \tau' \in \ZZ[\tau]$. Since $M$ has integer entries, the second eigenvalue is the algebraic conjugate $\lambda_2 = m \tau'+n = k-l\tau$. Then,
\[
k+n= \tr(M)= \lambda_1+\lambda_2= (m \tau +n)+(m \tau'+n)=(k-l \tau')+(k-l \tau) \,.
\]
It follows that $k+n=2n+m=2k-l$ and hence $k=n+l$ and $l=m$. Therefore,
\[
M= \begin{bmatrix} m+n & m \\ m & n\end{bmatrix},
\]
which is the desired result.
\end{proof}

\begin{example} When $m=f_k, n=f_{k-1}$ are consecutive Fibonacci numbers, we have
\[
M=\begin{bmatrix} m+n & m \\ m & n\end{bmatrix}= \begin{bmatrix} f_{k+1} & f_k \\ f_k & f_{k-1}\end{bmatrix}= \begin{bmatrix}
1 & 1 \\
1 & 0
\end{bmatrix}^k \,.
\]
The ring $\ZZ[\tau]$ is a free $\ZZ$-module with basis $\{ \tau, 1\}$. Now, each element $m\tau+n$ induces a $\ZZ$-linear homomorphism $T_{m\tau+n}: \ZZ[\tau]\to \ZZ[\tau]$. The matrix $M$ is exactly the matrix of the linear mapping $T_{m\tau+n}$ with respect to the canonical basis $\{ \tau, 1 \}$. The product of the matrices $M$ and 
\[
M'= \begin{bmatrix} m'+n' & m' \\ m' & n'\end{bmatrix}
\] 
has $[\tau\ 1]^T$ as a right eigenvector, and hence $M\cdot M'$ also be of this form. \exend
\end{example}

\begin{lemma}\label{lem2a1} Let $\varrho : \Sigma \to \Sigma^*$ be any primitive substitution, $\ttw \in \W$ such that $\varrho(\ttw)=\ttw$, and $M_\varrho$ be the substitution matrix of $\varrho$. The following are equivalent.
\begin{itemize}
\item[(i)] $\lim_{n\to\infty} {r(n)}/{n}=1$.
\item[(ii)]  $[\tau\ 1]^T$ is a right Perron--Frobenius eigenvector for $M_\varrho$.
\item[(iii)] There exist $m \in \NN$ and $n \in \ZZ_{\geqslant 0}$ such that
\[
M_\varrho= \begin{bmatrix} m+n & m \\ m & n\end{bmatrix}.
\]
\end{itemize}
\end{lemma}

\begin{proof}
Let $[u\ 1]^T$ be a right Perron--Frobenius eigenvector of $M_\varrho$. By Lemma~\ref{sub:asympt}, we have
\[
\lim_{n\to\infty} \frac{r(n)}{n}=\frac{u^2-1}{u}.
\]

Now, (i) holds if, and only if, $\frac{u^2-1}{u} =1$, so $u \in \{ \tau, \tau' \}$. Since the right Perron--Frobenius eigenvector is positive, (i) holds if and only if $u=\tau$. This shows (i)$\Leftrightarrow$(ii).

To see that (ii)$\Rightarrow$(iii), by Lemma~\ref{lem2}, there exists some $m,n \in \ZZ$ such that
\[
M_\varrho= \begin{bmatrix} m+n & m \\ m & n \end{bmatrix}.
\]
Since $M_\varrho$ is a substitution matrix, $m,n \geqslant 0$. It is clear that $M_\varrho$ is primitive if, and only if, $m \neq 0$.

Finally, we note that (iii)$\Rightarrow$(ii) follows from Lemma~\ref{lem2}.
\end{proof}

\begin{example}For each $m \in \NN$ and $n \in \ZZ_{\geqslant 0}$, the substitution
\[
\varrho:\begin{cases}
  \tta \to\tta^{m+n}\ttb^m \\
  \ttb \to \tta^m\ttb^n
\end{cases}
\]
is a primitive substitution with substitution matrix
\[
M_\varrho= \begin{bmatrix} m+n & m \\ m & n \end{bmatrix} \,.
\]
This family of substitutions, and all their permutations for which $\varrho(\tta)$ starts with $\tta$, give all the substitutions with fixed word $\ttw \in \Wa$ satisfying
\begin{equation}\label{eq-r1}
\lim_{n\to\infty} \frac{r(n)}{n} =1 \,,
\end{equation}
while all the permutations of these $\varrho(\tta)$ and $\varrho(\ttb)$ with the property that $\varrho(\ttb)$ starts with $\ttb$ yield all the substitutions with fixed points $\ttw \in \Wb$ satisfying \eqref{eq-r1}.\exend
\end{example}

Now, we discuss the binary words satisfying $r(n)=n+1$, and then the generalization $r(n)=n+j$ for any $j$. Since $r(n)$ does not vanish, we must have $j \geqslant0$. On another hand, for $j \geqslant 0$, the sequence $r(n)=n+j$ is a strictly increasing sequence of positive integers, and hence there exists a unique $\ttw$ such that $r(n)=n+j$.

We start this discussion with the following result that follows from Prop~\ref{p1}:

\begin{theorem}\label{thm:fibplusone}
Let $\tt f$ be the Fibonacci word and $\ttw = D (\tt f)$ be the Fibonacci word with the initial $\tta\ttb$ deleted. Then, $r_\ttw(n)=n+1$ for all $n$.
\qed
\end{theorem}

The following result is a corollary of Corollary~\ref{cor2}.

\begin{theorem}\label{thm:fibk} Let $\ttw_k = D^k(\ttf)$, the word obtained from the Fibonacci by deleting the first $k$ $\tta$'s and the first $k$ $\ttb$'s. Then, there exists some $N=N(k)$ such that, for all $n>N$, $r_{\ttw_k}(n)=n+k$.\qed
\end{theorem}

Now, the word $D^2(\ttf)$ has $r(n)=n+2$ for all $n \geqslant 2$ and $r(1)=2$. Moreover, we show below that $D^k(\ttf)$ is not a fixed point of a primitive substitution for any $k\geqslant 1$.

We now show that $\ttw= D(\ttf)$  can be obtained from Fibonacci via a different process. In particular, we get that two completely unrelated processes applied to the Fibonacci word lead to the same equality.

\begin{theorem}[The Fibonacci Switch]
Let $\varrho: \Sigma \to \Sigma^{*}$ be defined by $\varrho(\tta)=\tta\tta\ttb$ and $\varrho(\ttb)=\tta\ttb$.
Next, split the Fibonacci word $\ttf$ into level-$2$ supertiles, replace each level-$2$ supertile $\ttA_2=\tta\ttb\tta$ by $\varrho(\tta)=\tta\tta\ttb$, and keep each level-$2$ supertile $\ttB_2=\tta\ttb= \varrho(\ttb)$ unchanged. Let
\[
\ttw=\underbrace{\tta\tta\ttb}_{\varrho(\tta)}\underbrace{\tta\ttb}_{\varrho(\ttb)}\underbrace{\tta\tta\ttb}_{\varrho(\tta)} \underbrace{\tta\tta\ttb}_{\varrho(\tta)}\underbrace{\tta\ttb}_{\varrho(\ttb)}\underbrace{\tta\tta\ttb}_{\varrho(\tta)}\underbrace{\tta\ttb}_{\varrho(\ttb)}\underbrace{\tta\tta\ttb}_{\varrho(\tta)}  \cdots
\]
be the word obtained via this Fibonacci switch. Then $\ttw = \varrho(\ttf) = D(\ttf)$.
\end{theorem}

\begin{proof}
Consider the $n$-th letter in the Fibonacci word $\ttf$. Let us compare the switches in
\begin{align*}
\ttw&=\underbrace{\tta\tta\ttb}_{\varrho(\tta)}\underbrace{\tta\ttb}_{\varrho(\ttb)}\underbrace{\tta\tta\ttb}_{\varrho(\tta)} \underbrace{\tta\tta\ttb}_{\varrho(\tta)}\underbrace{\tta\ttb}_{\varrho(\ttb)}\underbrace{\tta\tta\ttb}_{\varrho(\tta)}\underbrace{\tta\ttb}_{\varrho(\ttb)}\underbrace{\tta\tta\ttb}_{\varrho(\tta)}  \cdots \\
\ttf&= \underbrace{\tta\ttb\tta}_{\ttA_2}\underbrace{\tta\ttb}_{\ttB_2}\underbrace{\tta\ttb\tta}_{\ttA_2} \underbrace{\tta\ttb\tta}_{\ttA_2}\underbrace{\tta\ttb}_{\ttB_2}\underbrace{\tta\ttb\tta}_{\ttA_2}\underbrace{\tta\ttb}_{\ttB_2}\underbrace{\tta\ttb\tta}_{\ttA_2}  \cdots
\end{align*}
Here, the supertiles we are comparing are
\[
\ttA_2=\tta\ttb\tta \leftrightarrow \varrho(\tta) = \tta\tta\ttb \qquad\mbox{and}\qquad
\ttB_2= \tta \ttb  \leftrightarrow \varrho(\ttb) = \tta\ttb.
\]
Each of these four supertiles contains exactly one $\ttb$, and the order of $\ttA_2$ and $\ttB_2$ ($\varrho(\tta)$ and $\varrho(\ttb)$) is the same as the order of letters in the Fibonacci word. Whenever an $\ttA_2$ supertile appears, then the position of the corresponding $\ttb$ in $\varrho(\tta)$ increases by 1. Whenever a $\ttB_2$ supertile appears, then the position of the corresponding $\ttb$ in $\varrho(\ttb)$ stays the same. It follows that
\[
p_{\ttb,\ttw}(n) =
\begin{cases}
p_{\ttb, \ttf}(n)+1    &\mbox{if the } n \mbox{th letter in the Fibonacci word is } \tta    \\
p_{\ttb, \ttf}(n)     &\mbox{if the } n \mbox{th letter in the Fibonacci word is } \ttb   \,.
\end{cases}
\]
Next, recall that
\[
\ttA_2=\ttA_1\ttB_1=(\tta \ttb) \tta    \leftrightarrow \rho(\tta) = \tta\tta\ttb \qquad\mbox{and}\qquad
\ttB_2=\ttA_1 =\tta \ttb    \leftrightarrow \rho(\ttb) = \tta\ttb.
\]
Looking at the position changes of $\tta$s, we see immediately that the $\tta$ in $\ttB_1$ moves one position back in $\ttw$, and the $\tta$ in $\ttA_1$ stays in the same position. Thus,
\[
p_{\tta,\ttw}(n) =
\begin{cases}
p_{\tta, \ttf}(n) &\mbox{if the } n \mbox{th letter in the Fibonacci is } \tta \\
p_{\tta, \ttf}(n) -1 &\mbox{if the } n \mbox{th letter in the Fibonacci is } \ttb   \,.
\end{cases}
\]
This implies that $\ttw$ has relative position function $r_{\ttw}(n)=r_{\ttf}(n)+1 =n+1$. Thus, by uniqueness and Theorem \ref{thm:fibplusone}, $\ttw = \rho(\ttf) = D(\ttf)$.
\end{proof}

We now prove that combining the property of aperiodicity/non-triviality with the generalization of the properties in  Theorem~\ref{thm-fib}(a) and Theorem~\ref{thm-fib}(b) uniquely identifies the Fibonacci word $\ttf$.

\begin{theorem}\label{fib-uni} Let $\ttw \in \Wa$ be a word. Then, both $\Delta p_\tta$ and $\Delta p_\ttb$ are factors of $\ttw$ if and only if either $\ttw=\ttf$, or $\ttw=(\tta\ttb)^\omega$.

In particular, $\ttf$ is the only word in $\Wa$ which is isomorphic to both $\Delta p_\tta$ and $\Delta p_\ttb$.
\end{theorem}

\begin{proof}
Before proceeding, let us summarize our strategy. First, we will show that any word starting with $\tta \tta$ and satisfying the given conditions must consist of $\tta$'s only. Therefore, we can focus on words starting with $\tta \ttb$.

This means that $\Delta p_\tta(1)$ and $\Delta p_\tta(2)$ uniquely identify $k$ and $l$ and $\Delta p_\ttb(1)$ and $\Delta p_\ttb(2)$ uniquely identify $m$ and $n$. To obtain one of these, we need to look at the first few bits of $\ttw$. Finally, knowing one of the pairs $(k,l)$ or $(m,n)$ and that $\ttw \in \tta \ttb \W$ we can reconstruct $\ttw$. Let us proceed along these lines.

Toward a contradiction, assume that $\ttw$ starts with $\tta \tta$. Then, under the coding $\tta \to k$, the sequence $\Delta p_\tta$ starts with $k,k$. Since $\Delta p_\tta(1)=1$, we get that $k=1$ and hence $\Delta p_\tta$ starts with $1,1$.
A short induction proves that for each $n$, $\ttw$ starts with $\tta^n$, and so  $\Delta p_\tta(1)$ starts with $n$ repeated $1$s.
Therefore, $\ttw = \tta^\omega \notin \Wa$, a contradiction.

Now, we look at six cases.

\underline{Case 1:} Suppose $\ttw=\tta \ttb \tta \tta \cdots$. Under the coding $(\tta,\ttb)=(k,l)$, the sequence $\Delta p_\tta$ starts with $k,l,k,k$. Since $\Delta p_\tta(1)=2$ and  $\Delta p_\tta(2)=1$, we get that $(k,l)=(2,1)$. For simplicity, let $\ttw = \ell_0 \ell_1 \ell_2 \cdots$ and $\ttf = f_0 f_1 f_2 \cdots$. We prove by induction that $f_j = \ell_j$ for all $j \in \ZZ_{\geqslant 0}$. By hypothesis, $f_j = \ell_j$ for $0\leqslant j\leqslant 3$. Next, suppose $f_j = \ell_j, 0\leqslant j\leqslant r$. Then, $\Delta p_{\tta,\ttw}(j) = \Delta p_{\tta,\ttf}(j)$, for $0\leqslant j\leqslant r$. As $p_{\tta, \ttw}(1)=0=p_{\tta, \ttf}(1)$, it follows that the positions of the first $r+1$ $\tta$'s in $\ttw$ and $\ttf$ agree. Hence, the bits of $\ttw$ and $\ttf$ are the same up to position $p_{\tta,\ttw}(r+1)=p_{\tta,\ttf}(r+1)$. Now, since $p_{\tta,\ttf}(2)=2$, we trivially get that $p_{\tta, \ttf}(r+1)=p_{\tta,\ttw}(r+1)\geqslant r+1$. Therefore $\ell_{r+1}=f_{r+1}$.

\underline{Case 2:} Suppose that $\ttw=\tta \ttb \tta \ttb \tta\cdots$. Then, the sequence $\Delta p_\tta$ starts with $k,l,k,l,k$ and hence $k=l=2$. This shows that $\Delta p_\tta(j)=2$ for all $j\geqslant1$. Thus $\ttw = (\tta \ttb)^\omega$.

\underline{Case 3:} Suppose $\ttw=\tta \ttb \tta \ttb \ttb\cdots$.  Then, the sequence $\Delta p_\ttb$ starts with $m,n,m,n,n$, which gives $(m,n)=(2,1)$ and hence $\Delta p_\ttb$ starts with $2,1,2,1,1$. Since $p_\ttb(1)=1$, the first six positions of $\ttb$ in $\ttw$ are $1,3,4,6,7,8$ and hence $\ttw = \tta \ttb \tta \ttb \ttb \tta \ttb \ttb \ttb \cdots$, which does not yet cause a problem---we need to go a bit further. Under the coding $\tta \to m=2$ and $\ttb \to n =1$, we get that $\Delta p_\ttb$ starts with $2,1,2,1,1,2,1,1,1$, and hence, the first ten positions of $\ttb$ in $\ttw$ are $1,3,4,6,7,8, 10, 11, 12, 13$, and hence $\ttw = \tta \ttb \tta \ttb \ttb \tta \ttb \ttb \ttb \tta \ttb \ttb \ttb \ttb \cdots$. This implies that $\Delta p_\tta$ starts with $2, 3, 4$, which contradicts the hypothesis that $\Delta p_\tta$ is obtained from $\ttw$ by under a coding $(\tta,\ttb)=(k,l)$.

\underline{Case 4:} Suppose $\ttw=\tta \ttb \ttb \tta \tta\cdots$. Here, $\Delta p_\tta(1)=3$ and $\Delta p_\tta(2)=1$, which gives that $(k,l)=(3,1)$, hence $\Delta p_\tta$ starts with $3,1,1,3,3$. Since $p_\tta(1)=0$, the first six positions of $\tta$ in $\ttw$ are $0,3,4,5,8, 11$, and so $\ttw = \tta \ttb \ttb \tta \tta \tta \ttb \ttb \tta \ttb \ttb \tta \cdots$. This immediately implies that $\Delta p_\ttb$ starts with $1, 4 , 1, 2,$ which contradicts the fact that $\Delta p_\ttb$ is obtained from $\ttw$ by under a relabeling $(\tta,\ttb)=(m,n)$.

\underline{Case 5:} Suppose $\ttw=\tta \ttb \ttb \tta \ttb\cdots$. In this case, $\Delta p_\ttb(1)=1$ and $\Delta p_\ttb(2)=2$, which implies that $(m,n)=(1,2)$, and hence $\Delta p_\ttb$ starts with $1,2,2,1,2$. Since $p_\ttb(1)=1$, the first six positions of $\ttb$ in $\ttw$ are $1,2,4,6,7,9$ and so $\ttw = \tta \ttb \ttb \tta \ttb \tta \ttb \ttb \tta \ttb \cdots$. This immediately implies that $\Delta p_\tta$ starts with $3, 2, 3$, which contradicts the fact that $\Delta p_\tta$ is obtained from $\ttw$ by under a coding $(\tta,\ttb)=(k,l)$.

\underline{Case 6:} Finally, suppose that $\ttw$ starts with $\tta \ttb^3$. This implies that $\Delta p_\ttb$ starts with $1,1$. Therefore, $\Delta p_\ttb$ is obtained from $\ttw$ under the coding $(\tta,\ttb)=(1,1)$, which implies that $\Delta p_\ttb(n)=1$ for all $n$. In particular, $\tta$ appears in $\ttw$ only on the first position, which contradicts $\ttw \in \Wa$.

This completes the proof.
\end{proof}

\begin{remark} For each pair $(k,l)\in \NN \times \NN$ with $k\neq 1$, there exists a unique word $\ttw_{k,l} \in \tta\mathcal{W}$ with the property that $\Delta p_\tta$ is obtained from $\ttw_{k,l}$  under the coding $(\tta,\ttb)=(k,l)$.
\end{remark}

By combining Theorem~\ref{fib-uni} with the reflection operator, we obtain the following corollary.

\begin{corollary} Let $\ttw \in \Wb$ be a word. Then, both $\Delta p_\tta$ and $\Delta p_\ttb$ are factors of $\ttw$ if and only if either $\ttw=\bar{\ttf}$, or $\ttw=(\ttb\tta)^\omega$.

In particular, $\bar{\ttf}$ is the only word in $\Wb$ which is isomorphic to both $\Delta p_\tta$ and $\Delta p_\ttb$.\qed
\end{corollary}

Since the Fibonacci word $\ttf$ is isomorphic via reflection to $\bar{\ttf}$, we have the following corollary.

\begin{corollary} The Fibonacci word $\ttf$ is, up to isomorphism, the only word on a two letter alphabet which is isomorphic to the difference sequence of the position function of each letter.\qed
\end{corollary}

Now, we consider a variant of the Fibonacci substitution.

\begin{definition} We call the substitution,
\[
\varrho_F':\begin{cases}
\tta\to \ttb\tta \\
\ttb\to\tta \,,
\end{cases}
\]
the \emph{backwards Fibonacci substitution}, or the \emph{iccanobiF substitution}.
\end{definition}

The substitution $\varrho_F'$ does not have a fixed point, but $(\varrho_F')^2$ has two fixed points, one starting with $\tta$, and one starting with $\ttb$. We will look at the fixed points of $(\varrho_F')^2$ and their relative position functions. To explicitly calculate these fixed points, we relate $\varrho_F'$ to the Fibonacci substitution $\varrho_F$.

By \cite[Remark~4.6]{TAO}, the substitutions $\varrho_F$ and $\varrho_F'$ are conjugate. More precisely, for all words $\ttw\in\Sigma^*$, we have $\varrho_F(\ttw) \tta = \tta \varrho_F'(\ttw)$. Our immediate goal is to show that we can relate $\varrho_F^n$ and $(\varrho_F')^n$ via conjugation relations. Noting that for $n\geqslant 1$, we have $\varrho_F^n(\ttb)= \varrho_F^{n-1}(\tta)$ and $(\varrho_F')^n(\ttb)= (\varrho_F')^{n-1}(\tta)$, it suffices to relate $\varrho_F^n(\tta)$ to $(\varrho_F')^n(\tta)$. We achieve this in the following result.

\begin{proposition}\label{prop:rev-fib} For all $n \geqslant 1$, we have both
\begin{itemize}
\item[(a)] $\tta \ttb \varrho_F^{2n}(\tta)=(\varrho_F')^{2n}(\tta) \ttb \tta$, and
\vspace{.1cm}
\item[(b)] $\ttb \tta \varrho_F^{2n-1}(\tta)=(\varrho_F')^{2n-1}(\tta) \tta \ttb$.
\end{itemize}
\end{proposition}

\begin{proof}
We prove both by induction.

{\bf (a)} For $n=1$, we have $\tta \ttb \varrho_F^{2}(\tta)=\tta \ttb \tta \ttb \tta = \left(\varrho_F'\right)^{2}(\tta) \ttb \tta$. Now, suppose the result holds for some $n\geqslant 1$. Then
\begin{align*}
\tta \ttb \varrho_F^{2n+2}(\tta)&=\tta \ttb \varrho_F^{2n}(\tta \ttb \tta ) =  \tta \ttb \varrho_F^{2n}(\tta) \varrho_F^{2n-1}(\tta) \varrho_F^{2n}(\tta) \\
&= (\varrho_F')^{2n}(\tta) \ttb \tta  \varrho_F^{2n-1}(\tta) \varrho_F^{2n}(\tta) = (\varrho_F')^{2n}(\tta)  (\varrho_F')^{2n-1}(\tta) \tta \ttb  \varrho_F^{2n}(\tta) \\
&= (\varrho_F')^{2n}(\tta)  (\varrho_F')^{2n-1}(\tta)  (\varrho_F')^{2n}(\tta)  \ttb \tta = (\varrho_F')^{2n}(\tta \ttb \tta) \ttb \tta = (\varrho_F')^{2n+2}(\tta) \ttb \tta \,.
\end{align*}

{\bf (b)} For $n=1$, we have $\ttb \tta \varrho_F^{1}(\tta)=\ttb \tta \tta \ttb = \left(\varrho_F'\right)^{1}(\tta) \tta \ttb$. Now, suppose the result holds for some $n\geqslant 1$. Then
\begin{align*}
\ttb \tta \varrho_F^{2n+1}(\tta)&=\ttb \tta \varrho_F^{2n-1}(\tta \ttb \tta ) =  \ttb \tta \varrho_F^{2n-1}(\tta) \varrho_F^{2n-2}(\tta) \varrho_F^{2n-1}(\tta) \\
&= (\varrho_F')^{2n-1}(\tta) \tta \ttb  \varrho_F^{2n-2}(\tta) \varrho_F^{2n-1}(\tta) = (\varrho_F')^{2n-1}(\tta)  (\varrho_F')^{2n-2}(\tta) \ttb \tta  \varrho_F^{2n-1}(\tta) \\
&= (\varrho_F')^{2n-1}(\tta)  (\varrho_F')^{2n-2}(\tta)  (\varrho_F')^{2n-1}(\tta)  \tta \ttb = (\varrho_F')^{2n}(\tta \ttb \tta) \ttb \tta = \left(\varrho_F'\right)^{2n+2}(\tta) \tta \ttb \,.\qedhere
\end{align*}
\end{proof}

Similarly, we have the following result.

\begin{proposition} For all $n \geqslant 1$, we have both
\begin{itemize}
\item[(a)] $\tta \ttb \varrho_F^{2n+1}(\ttb)=\left(\varrho_F'\right)^{2n+1}(\ttb) \ttb \tta$, and
\vspace{.1cm}
\item[(b)] $\ttb \tta \varrho_F^{2n}(\ttb)=\left(\varrho_F'\right)^{2n}(\ttb) \tta \ttb$.\qed
\end{itemize}
\end{proposition}

As a consequence, we can understand the fixed points of $(\varrho_F')^2$.

\begin{theorem} Let $\varrho_F'$ be the iccanobiF substitution and $\ttf$ be the Fibonacci word. Then, the one-sided fixed points of $(\varrho_F')^{2}$ (starting with $\tta$ or $\ttb$) exist, and satisfy
\[
\lim_{n\to\infty} (\varrho_F')^{2n} (\tta) ={\rm Pre}_{\tta \ttb } (\ttf) \in \Wa\qquad\mbox{and}\qquad
\lim_{n\to\infty} (\varrho_F')^{2n} (\ttb) ={\rm Pre}_{\ttb \tta } (\ttf)  \in \Wb\,.
\]
Moreover, we have both
\[
r_{{\rm Pre}_{\tta \ttb } (\ttf)}(n)= \begin{cases}
    1 & \mbox{ if } n=1  \\
    n-1 & \mbox{ if } n>1
\end{cases}
\qquad
\mbox{and}
\qquad
r_{{\rm Pre}_{\ttb \tta } (\ttf)}(n)=
\begin{cases}
    -1 & \mbox{ if } n=1  \\
    n-1 & \mbox{ if } n>1\,.
\end{cases}
\]
\end{theorem}

\begin{proof}
By Proposition~\ref{prop:rev-fib} we have $(\varrho_F')^{2n}(\tta) \ttb \tta = \tta \ttb \varrho_F^{2n}(\tta)$. Letting $n \to \infty$, we get $\tta \ttb \ttf = {\rm Pre}_{\tta \ttb}(\ttf)$.
Similarly, $(\varrho_F')^{2n-1}(\tta) \tta \ttb = \ttb \tta \varrho_F^{2n-1}(\tta)$, which gives $\ttb \tta \ttf = {\rm Pre}_{\ttb \tta}(\ttf)$. The remaining claims follow immediately.
\end{proof}

Let us note here the following result which shows that we can get from the Fibonacci substitution words $\ttw$ for which, we eventually have $r_\ttw(n)=n+j$ for all values of $j$.

\begin{theorem} Let $w$ be any balanced word of length $2j$ and let $\ttw_j = {\rm Pre}_{w} (\ttf)$, the word obtained from the Fibonacci by the prefix $w$. Then, for all $n>j$ we have $r_{\ttw_j}(n)=n-j$.

Moreover, if $w=(\tta \ttb)^j$, then $r_{\ttw_j}(1)=r_{\ttw_j}(2)=\cdots =r_{\ttw_j}(j)=1$.\qed
\end{theorem}

We now turn to the one-sided fixed points of the family of substitutions $\sigma$ on $\Sigma$ with the property that $\sigma(\tta)$ and $\sigma(\ttb)$ contain one $\ttb$ in total. We will add two further restrictions on this family. First, if the $\ttb$ appears in $\sigma(\ttb)$, then $\tta \to \sigma(\tta)$ generates the 1-sided word containing only $\tta$, which is not interesting. Therefore, we assume that the single $\ttb$ appears in $\sigma(\tta)$. In this case, the only way of getting a fixed point for $\sigma$, and not only for one of its larger powers, is if $\sigma(\tta)$ starts with $\tta$---this is our second restriction. Therefore, we consider the following family of substitutions.

\begin{definition}The \emph{extended Pisa family of substitutions} is given by $\sigma_{k,l,m}$ with $k, m \geqslant 1$ and $l \geqslant 0$, where
\[
\sigma_{k,l,m}:
\begin{cases}\tta\to\tta^k\ttb\tta^l  \\
\ttb\to\tta^m \,.
\end{cases}
\]
\end{definition}

Note that, in the case $l=0, m=1$, the substitutions $\sigma_{k,0,1}$ are called the \emph{noble means substitutions} in \cite[Rem.~4.7]{TAO} and \cite{EMM}, and $\sigma_{1,0,2}$ is the \emph{period-doubling substitution} \cite[Sect.~4.5.1]{TAO}.

For general $k$, $l$ and $m$, the substitution matrix is
\[
M=M_{\sigma_{k,0,1}}= \begin{bmatrix}
k+l & 1 \\
m   &0
\end{bmatrix} \,.
\]
Since $M^2>0$, $M$ is a primitive matrix. Its two distinct eigenvalues are
\[
\lambda_{\pm}= \frac{(k+l) \pm  \sqrt{(k+l)^2+4m}}{2} \,,
\]
which only depend on the two parameters $k+l$ and $m$. Since the product of eigenvalues is $\det(M)=-m<2$, we have $\lambda_2 <0$. Thus, the substitution is Pisot (see \cite{Sing} for definition and properties) if, and only if, $m < (k+l)+1$. When $m=(k+l)+1$, the eigenvalues are $\lambda_1=(k+l)+1$ and $\lambda_2=-1$.

We can now prove the following result, which significantly extends the results on the Fibonacci substitution---$(k,l,m)=(1,0,1)$.

\begin{theorem}\label{thm:6} Let $k, m \geqslant 1$, $l \geqslant 0$, and $\ttw$ be the one-sided fixed point of $\sigma_{k,l,m}$. Then,
\[
p_{\ttb}(n)= m\cdot p_\tta(n) + (k+l+1-m)  n + m-l-1\,.
\]
In particular,
\begin{equation}\label{eq-r}
r(n)= (m-1) \cdot p_\tta(n) + (k+l+1-m)  n + m-l-1.
\end{equation}
Moreover, $\ttw$ is the unique word satisfying \eqref{eq-r}.
\end{theorem}

\begin{proof} We split $\ttw$ into level-$1$ supertiles $\ttA=\tta^k\ttb\tta^l$ and $\ttB=\tta^m$, so that
\[
\ttw =\underbrace{\tta \cdots }_{X_0=\ttA} \underbrace{\tta \cdots }_{X_1}  \underbrace{\tta \cdots }_{X_2} \ \cdots,
\]
where $X_i\in\{\ttA,\ttB\}$ for each $i\geqslant 0$. To calculate the position $p_\ttb(n)$ of the $n$-th $\ttb$, we note that each $\ttA$ contains exactly one $\ttb$ and each $\ttB$ contains no $\ttb$. This means that the $n$-th $\ttb$ appears in the $n$-th $\ttA$ supertile. For simplicity, denote $j:=p_\tta(n)$, so that the $n$-th $\ttb$ appears inside $X_j$. Now, there are $j$ supertiles before $X_j$. Since $X_j$ is the $n$-th $\ttA$ supertile, there are exactly $n-1$ $\ttA$ supertiles before $X_j$. The remaining $j-n+1$ supertiles are $\ttB$ supertiles. Since each $\ttA$ supertile contains $k+l+1$ letters and each $\ttB$ supertile contains $m$ letters, there are
\[
(n-1) (k+l+1)+(j-n+1) m
\]
letters before the supertile $X_j$. Further, the single $\ttb$ is in position $k+1$ inside $X_j$. Remembering that our index count starts at $0$, for $\ttw=\ell_0\ell_1 \cdots$, we have
\begin{align*}
p_\ttb(n)&= (n-1) (k+l+1)+(j-n+1) m+ k   \\
&=n (k+l+1)+(p_\tta(n)-n+1) m+ k -(k+l+1)\\
&= m \cdot p_\tta(n) + (k+l+1-m)  n + m-l-1 \,,
\end{align*}
which proves the first claim as well as \eqref{eq-r}.

Lastly, we prove uniqueness of $\ttw$ satisfying \eqref{eq-r}. Note that this is not a trivial fact, since, in general, $r(n)$ uniquely determines $\ttw$ only when it is a specific function of $n$. Suppose $\ttw' \in \mathcal W$ has a relative position function satisfying (\ref{eq-r}) for some $k,m\geqslant 1$ and $l\geqslant 0$. Since $p_{\tta,\ttw'}(n) \geqslant n$,
\begin{align*}
p_{\ttb,\ttw'}(n) &= m \cdot p_{\tta,\ttw'}(n) + (k+l+1-m)  n + m-l-1 \\
&\geqslant p_{\tta,\ttw'}(n) + (m-1)n + (k+l+1-m)n + m-l-1 \\
& = p_{\tta,\ttw'}(n) + kn + ln + m -l -1 \\
& = p_{\tta,\ttw'}(n) + kn + l(n-1) + (m-1) > p_{\tta,\ttw'}(n)\,.
\end{align*}
This implies that $p_{\tta,\ttw'}(1) = 0$. Then, we reconstruct in the straightforward manner, where the $(n+1)$-th $\tta$ is placed in the first unoccupied spot and the $(n+1)$-th $\ttb$ is placed further along according to the formula. Thus, there is only one word that is constructed from these formulas, which must be the fixed point of $\sigma_{k,l,m}$.
\end{proof}

The following is immediate.

\begin{corollary} Let $k, m \geqslant 1$, $l \geqslant 0$, and $\ttw$ be the one-sided fixed point of $\sigma_{k,l,m}$. Then,
\begin{equation*}
 r(n) = \left(\frac{m-1}{m}\right) p_{\ttb}(n) +\left(\frac{k+l+1-m}{m}\right) n +\frac{ m-l-1}{m} \,.\qed
\end{equation*}
\end{corollary}

\medskip

Now we will illustrate this result with a few examples.

\begin{example}
Recall from above, the family of noble means substitutions are given by $\sigma_{k,0,1}$. The members of this family behave like the Fibonacci substitution. In particular, for each fixed point, we have $r(n) = kn$. \exend
\end{example}

\begin{example}
The period-doubling substitution $\varrho_{\rm pd}$ is given by
\[
\varrho_{\rm pd}:=\sigma_{1,0,2}:
\begin{cases}
    \tta \rightarrow \tta \ttb \\
    \ttb \rightarrow \tta \tta\, ,
\end{cases}
\]
which has the unique one-sided fixed word
\[
\ttw = \lim_{n\to\infty}\varrho_{\rm pd}^n(\tta)=\lim_{n\to\infty} \sigma_{1,0,2}^n(\tta) = \tta\ttb\tta\tta\tta\ttb\tta\ttb\tta\ttb\tta\tta\cdots.
\]
By Theorem \ref{thm:6}, we have $p_{\ttb}(n) = 2p_{\tta}(n) + 1$, and
\[
r(n) = p_{\tta}(n)+1 = \frac{1}{2}\cdot p_{\ttb}(n) + \frac{1}{2}.
\]
Analogous to the noble means family, the period-doubling substitution and its fixed point is part of a well-behaved family, specifically,
\[
\sigma_{k,0,k+1}: \begin{cases}
\tta \to \tta^{k}\ttb\\
\ttb \to \tta^{k+1}\, .
\end{cases}
\] where $p_{\ttb}(n) = (k+1)p_{\tta}(n) + 1$ and $r(n) = k\cdot p_{\tta}(n)+1 = \left(\frac{k}{k+1}\right)p_{\ttb}(n) + \frac{k}{k+1}.$\exend
\end{example}

\begin{example}
The ``mixture/addition" of Fibonacci and period-doubling is most interesting. Here, let $\ttw$ be the one-sided fixed point of
\[
\sigma_{2, 0, 2}
:\begin{cases}
\tta\to \tta \tta \ttb \\
\ttb\to\tta \tta.
\end{cases}
\]
Then, by Theorem~\ref{thm:6}, $p_\ttb(n) = 2\cdot p_\tta(n) + n + 1$ and $r(n) = p_{\tta}(n)+ n+ 1$. Further, since each level-$1$ supertile contains exactly two $\tta$'s, which are consecutive, we can relabel this substitution using $(\alpha,\beta):=(\tta\tta,\ttb)$. The relabeled fixed word is the fixed word $\ttw'$ of the substitution $\sigma'$ satisfying
\begin{align*}
\sigma'(\alpha)&= \sigma_{2,0,2}(\tta\tta) = \tta\tta\ttb\tta\tta\ttb =  \alpha \beta \alpha \beta \\
\sigma'(\beta)&= \sigma_{2,0,2}(\ttb) = \tta\tta = \alpha\,.
\end{align*}
Moreover, the relation $p_{\ttb, \ttw}(n)=2\cdot p_{\tta, \ttw}(n) + n+1$ implies that
\[
p_{\beta,\ttw'}(n)=p_{\tta, \ttw}(n)+n+1 \,,
\]
so that
\[
r_{\ttw}(n)=p_{\ttb, \ttw}(n)-p_{\tta,\ttw}(n)=p_{\beta, \ttw'} (n)
\]
which is an interesting relationship between these infinite words.\exend
\end{example}

To finish this section, we obtain the letter frequencies and the mean values of the the (relative) position functions for the entire extended Pisa family---which come as a generalization of the Fibonacci example. First, note that the Perron--Frobenius eigenvalue of $M_{\sigma_{k,l,m}}$, as mentioned above, is $\tau_{k+l,m}$ where
\[
\tau_{j,m}:=\frac{j + \sqrt{(j)^2+4m}}{2}
\]
and the right Perron--Frobenius eigenvector is $[\tau_{k+l,m} \
1]^T$. The characteristic equation of $M_{\sigma_{k,l,m}}$ is $X^2-(k+l)X-m=0$, so
\[
\tau_{k+l,m}^2=(k+l)\tau_{k+l,m}+m \,.
\]
In particular,
\[
\frac{1}{\tau_{k+l,m}}= \frac{\tau_{k+l,m}-(k+l)}{m} \,.
\]
A direct application of Lemma~\ref{sub:asympt} and \eqref{eq-r} gives
\[
  {\rm Freq}(\tta) =\frac{\tau_{k+l,m}}{\tau_{k+l,m}+1}=\frac{\tau_{j,m}-m}{k+l+1-m}, \qquad
  {\rm Freq}(\ttb) =\frac{1}{\tau_{k+l,m}+1}=\frac{k+l+1-\tau_{k+l,m}}{k+l+1-m} \,,
\]
\[
\lim_{n\to\infty} \frac{p_\tta(n)}{n}=1+\frac{\tau_{k+l,m}-(k+l)}{m}, \qquad
  \lim_{n\to\infty} \frac{p_\ttb(n)}{n}=\tau_{k+l,m}\,,
  \] and
  \[
  \lim_{n\to\infty} \frac{r(n)}{n}=\frac{(m-1)\tau_{k+l,m}+k+l}{m}= (m-1)\cdot \lim_{n\to\infty} \frac{p_\tta(n)}{n} + (k+l+1-m) \,.
\]

\medskip

Next, analogous to the proofs of Theorem~\ref{thm-2} and Lemma~\ref{lem2a1}, we establish the following.

\begin{proposition} Let $j$ and $m$ be two positive integers and let $\varrho$ be a primitive binary substitution with substitution matrix $M_\varrho$. Let $\ttw$ be a fixed word of the substitution. Then, the following are equivalent.
\begin{itemize}
  \item[(i)]  ${\rm Freq}(\tta) =\frac{\tau_{j,m}-m}{j+1-m} $.
  \vspace{.2cm}
  \item[(ii)] ${\rm Freq}(\ttb) =\frac{j+1-\tau_{j,m}}{j+1-m}$.
  \vspace{.2cm}
  \item[(iii)]$\displaystyle{\lim_{n\to\infty} \frac{p_\tta(n)}{n}= 1+\frac{\tau_{j,m}-j}{m}}$.
  \vspace{.2cm}
  \item[(iv)] $\displaystyle{\lim_{n\to\infty} \frac{p_\ttb(n)}{n}=\tau_{j,m}+1}$.
  \vspace{.2cm}
  \item[(v)]  $\displaystyle{\lim_{n\to\infty} \frac{r(n)}{n}=\frac{(m-1)\tau_{j,m}+j}{m}} $.
  \vspace{.2cm}
  \item[(vi)] $\displaystyle{\lim_{n\to\infty} \frac{r(n)}{n}= (m-1) \lim_{n\to\infty} \frac{p_\tta(n)}{n} + (j+1-m)}$.
  \vspace{.2cm}
  \item[(vii)] $[\tau_{j,m} \ 1]^T$ is a right eigenvector for $M_\varrho$.
  \end{itemize}
\end{proposition}

\begin{proof}
Let $[u \ 1]^T$ be the right Perron--Frobenius eigenvector for $M_\varrho$. Then, $u>0$. The discussion above shows that $u= \tau_{j,m}$ is the unique solution to each of the linear equations
\[
\frac{u}{u+1} =\frac{\tau_{j,m}}{\tau_{j,m}+1}=\frac{\tau_{j,m}-m}{j+1-m},\qquad
  \frac{1}{u+1}=\frac{1}{\tau_{j,m}+1}=\frac{j+1-\tau_{j,m}}{j+1-m},
  \]
  \[
  1+\frac{1}{u}=1+\frac{\tau_{j,m}-j}{m},\quad\mbox{and}\quad
  u+1=\tau_{j,m}+1.
\]
Applying Lemma~\ref{sub:asympt} proves the equivalence of (i), (ii), (iii), (iv) and (vii).

Next, the quadratic equation
\[
\frac{u^2-1}{u}= \frac{(m-1)\tau_{j,m}+j}{m} \\
\]
has $u= \tau_{j,m}$ as one of the solutions. Since the product of the solutions is $-1$, it follows that $u= \tau_{j,m}$ is the only positive solution. So, Lemma~\ref{sub:asympt} gives the equivalence of (v) and (vii).

Also by Lemma~\ref{sub:asympt}, we have
\[
\lim_{n\to\infty} \frac{r(n)}{n}= (m-1) \lim_{n\to\infty} \frac{p_\tta(n)}{n} + (j+1-m),
\] if, and only if
\[
\frac{u^2-1}{u}=(m-1)(1+\frac{1}{u})+ (j+1-m),
\] if, and only if,
\[
\frac{u^2-1}{u}=\frac{m-1}{u}+j
\] if, and only if, $u^2-ju-m=0$, which, again, has the unique positive solution $u= \tau_{j,m}$. This finishes the proof.
\end{proof}

Finally, we can prove the following result.

\begin{theorem}\label{thm:Mform} Let $j$ and $m$ be two positive integers and let $M \in M_2(\ZZ)$. Then,
\begin{itemize}
  \item[(a)] if $j^2+4m$ is not a perfect square, then $[\tau_{j,m} \
1]^T$ is a right eigenvector for $M$ if, and only if, there exist integers $s$ and $t$ such that
\[
M= \begin{bmatrix}
t+sj  &ms  \\
s&  t
\end{bmatrix},
\]
  \item[(b)] if $j^2+4m=r^2$ for some $r \in \ZZ$, then $[\tau_{j,m} \
1]^T$ is a right eigenvector for $M$ if, and only if, there exist integers $s$, $t$, $u$, and $v$ such that $u(j+r)^2+2(v-s)(j+r)-4t=0$ and
\[
  M =\begin{bmatrix} s & t \\  
  u & v \end{bmatrix}.
\]
\end{itemize}
\end{theorem}

\begin{proof}
\textbf{(a)} ($\Leftarrow$). A short computation yields
\begin{align*}
\begin{bmatrix}
t+sj  &ms  \\
s&  t
\end{bmatrix}
\begin{bmatrix}
\tau_{j,m} \\ 1
\end{bmatrix}&=\begin{bmatrix}
(t+sj)\tau_{j,m} +sm   \\s\tau_{j,m}+t
\end{bmatrix}=\begin{bmatrix}
(s(j\tau_{j,m}+m )+t\tau_{j,m})   \\
s\tau_{j,m}+t
\end{bmatrix}\\
&=\begin{bmatrix}
 (s\tau_{j,m}^2+t\tau_{j,m})   \\s\tau_{j,m}+t
\end{bmatrix}=\begin{bmatrix}
\tau_{j,m}(s\tau_{j,m}+t) \\ s\tau_{j,m}+t
\end{bmatrix}=( s\tau_{j,m}+t) \begin{bmatrix}
\tau_{j,m} \\ 1
\end{bmatrix},
\end{align*}
 and so, $[\tau_{j,m} \ 1]^T$ is a right eigenvector for $M$.

($\Rightarrow$). Since $\tau_{j,m} \notin \QQ$, $\QQ(\tau_{j,m})$ is a degree two extension of $\QQ$ and both (distinct) eigenvalues lie in this field and $[\tau_{j,m}' \ 1]^T$ is another (linearly independent eigenvector of $M$. Thus, for some integers $s$ and $t$, we have
\begin{align*}
M&= \begin{bmatrix}
    \tau_{j,m} & \tau_{j,m}' \\
    1 & 1
\end{bmatrix} \begin{bmatrix}
    t+s\tau_{j,m} & 0 \\
    0 & t+s \tau_{j,m}'
\end{bmatrix}\begin{bmatrix}
    \tau_{j,m} & \tau_{j,m}' \\
    1 & 1
\end{bmatrix}^{-1} \\
&=\frac{1}{(\tau_{j,m}-\tau_{j,m}')}\begin{bmatrix}
    (t+s\tau_{j,m} )\tau_{j,m} & (t+s \tau_{j,m}')\tau_{j,m}' \\
t+s\tau_{j,m}  & t+s \tau_{j,m}'
\end{bmatrix}  \begin{bmatrix}
    1 & -\tau_{j,m}' \\
    -1 & \tau_{j,m}
\end{bmatrix} \\
&=  \frac{1}{\tau_{j,m}-\tau_{j,m}'}
\begin{bmatrix}
    (t+s\tau_{j,m} )\tau_{j,m} -(t+s \tau_{j,m}')\tau_{j,m}'  & m(t+s\tau_{j,m}) - m(t+s \tau_{j,m}') \\
\left(( t+s\tau_{j,m})-(t+s \tau_{j,m}')\right)&  -(t+s\tau_{j,m}) \tau_{j,m}'   +(t+s \tau_{j,m}') \tau_{j,m}
\end{bmatrix} \\
&=\frac{1}{\tau_{j,m}-\tau_{j,m}'}
\begin{bmatrix}
( t+sj) (\tau_{j,m}  -\tau_{j,m}')  &ms (\tau_{j,m} - \tau_{j,m}') \\
s(\tau_{j,m}- \tau_{j,m}')&  t(\tau_{j,m}- \tau_{j,m}')
\end{bmatrix}=
\begin{bmatrix}
t+sj  &ms  \\
s&  t
\end{bmatrix}\,,
\end{align*}
which finishes the proof of (a).

\textbf{(b)}
We have $\tau_{j,m} ={(j + r)}/{2}$. A short straightforward calculation shows that $[(j+r)/2, \ 1]^T$ is a right eigenvector for
\[
M= \begin{bmatrix}
s & t \\
u & v
\end{bmatrix}\] if, and only if, $u(j+r)^2+2(v-s)(j+r)-4t=0$, which is the desired result.
\end{proof}

\section{Substitutions with exactly, and asymptotically, linear relative position}\label{sec:linear}

In this section, we first consider words $\ttw$ for which $r(n)$ is a linear function, then consider the asymptotically linear case. We start with the following result, which follows immediately from results of the previous sections.

\begin{theorem} The following hold.
\begin{itemize}
    \item[(a)] For all $j \geqslant 1$ the periodic word $\ttw =(\tta^j\ttb^j)^\omega$ satisfies $r(n)= j$.
\item[(b)] Let $k$ and $j$ be integers satisfying $0 \leqslant j \leqslant k-1$, and let $\ttw$ be the one-sided fixed point of the Pisa substitution $\sigma_{k-j, j, 1}$. Then $r(n)=k  n - j$.
\item[(c)] Let $k\geqslant 0$ be an integer and $\ttw$ be the one-sided fixed point of the Pisa substitution $\sigma_{1, k-1, 1}$. Then, $
r_{D(\ttw)}(n)=k  n +1$.\qed
\end{itemize}
\end{theorem}

The above result covers all the arithmetic progressions $kn+j$ for $-1\leqslant j \leqslant k-1$. Note that $kn+j$ cannot occur for $j=-k$, since this arithmetic progression contains the value $0$. In this case, the best we can hope is for $r(n)$ to eventually equal $kn+j$. This can always be achieved.

\begin{proposition}
Let $k\geqslant 0$ and $j$ be integers. Let $j=qk-r$ with $0 \leqslant r \leqslant k-1$ and $\ttw$ be one-sided fixed point of the Pisa substitution $\sigma_{1, k-1, 1}$.
\begin{itemize}
    \item[(a)]  If $q\geqslant 0$, then $r_{D^q(\ttw)}(n)= kn+j$ for all large enough $n$.
      \item[(b)]  If $q< 0$, then for each balanced word $s$ of length $-2q$, we have $r_{{\rm Pre}_{s}(\ttw)}(n)= kn+j$ for all large enough $n$. In particular, this holds when $s = (\tta \ttb)^{-q}$.\qed
\end{itemize}
\end{proposition}

We now turn to binary substitutions for which the relative position is asymptotically linear. In particular, analogous to the computations in Section~\ref{sect:fib}, we can find all binary substitutions whose one-sided fixed points satisfy
\[
\lim_{n\to\infty} \frac{r(n)}{n}= k \in \ZZ \,.
\]
As in the previous section, for each $k \in \NN$, denote $\tau_{k}:=\frac{k+ \sqrt{k^2+4}}{2}$. These numbers arise naturally as the eigenvalues of the substitution matrices of the Pisa substitutions $\sigma_{k-j,j,1}$. Of course, $\tau_1=\tau$ is the golden mean.
Recall that $\tau_k$ is a root of $X^2-kX-1=0$ and so, $\tau_k^2=k \tau_k+1$.

We can now prove the following, more general, version of Lemma~\ref{lem2}:

\begin{lemma}\label{lem2a} Let $k \in \NN$ and $M \in {M}_2(\ZZ)$. Then,  $[\tau_k\ 1]^T$ is a right eigenvector for $M$ if, and only if, there exist $m, n \in \ZZ$ such that
\[
M= \begin{bmatrix}
km+n & m \\
m & n
\end{bmatrix}\,.
\]
In this case, the eigenvalues are $n+m \tau_k$ and $n+m \tau'_k$, where $\tau_k'$ is the algebraic conjugate of $\tau_k$.
\end{lemma}
\begin{proof} ($\Leftarrow$). Note that
\[
\begin{bmatrix}
km+n & m \\
m & n
\end{bmatrix}\begin{bmatrix}
\tau_k  \\ 1
\end{bmatrix}=
\begin{bmatrix}
km\tau_k+n\tau_k+m \\
m \tau_k + n
\end{bmatrix}=\begin{bmatrix}
m \tau_k^2 + n\tau_k  \\ m \tau_k + n
\end{bmatrix}= \\
 (m \tau_k + n) \begin{bmatrix}
\tau_k  \\ 1
\end{bmatrix}
\] so that $[\tau_k\ 1]^T$ is a right eigenvector for $M$.

($\Rightarrow$).
Since $k>0$, the polynomial $X^2-kX-1$ is irreducible over $\QQ$ by the rational root test. In particular, $\QQ(\tau_k)$ is a degree-two extension of $\QQ$. We now mimic the proof of Theorem~\ref{thm:Mform}(a) to get
\begin{align*}
M&= \begin{bmatrix}
    \tau_k & \tau_k' \\
    1 & 1
\end{bmatrix} \begin{bmatrix}
    n+m\tau_k & 0 \\
    0 & n+m \tau_k'
\end{bmatrix}\begin{bmatrix}
    \tau_k & \tau_k' \\
    1 & 1
\end{bmatrix}^{-1} \\
&=\frac{1}{\tau_k-\tau_k'}\begin{bmatrix}
    (n+m\tau_k )\tau_k & (n+m \tau_k')\tau_k' \\
 n+m\tau_k  & n+m \tau_k'
\end{bmatrix}  \begin{bmatrix}
    1 & -\tau_k' \\
    -1 & \tau_k
\end{bmatrix} \\
&=\frac{1}{\tau_k-\tau_k'}
\begin{bmatrix}
(n+km )   (\tau_k-\tau_k') & m(\tau_k - \tau_k') \\
m(\tau_k - \tau_k')& n (\tau_k-\tau_k')
\end{bmatrix}
=\begin{bmatrix}
km+n & m \\
m & n
\end{bmatrix}\,,
\end{align*} which proves the result.
\end{proof}

We finish this section with the classification of all asymptotically linear relative position functions arising form primitive binary substitutions.

\begin{theorem}\label{thm-2} Let $\varrho$ be a primitive binary substitution with substitution matrix $M_\varrho$ and having one-sided fixed point $\ttw \in \W$. Let $k\in\NN$.
\begin{itemize}
  \item[(a)] The following are equivalent.
\begin{itemize}
\item[(i)] ${\lim_{n\to\infty} {r(n)}/{n}=k}$.
  \vspace{.2cm}
\item[(ii)]  $[\tau_k\ 1]^T$ is a right Perron--Frobenius eigenvector for $M_\varrho$.
\item[(iii)] There exist $m \in \NN$ and $n \in \ZZ_{\geqslant 0}$ such that
$
M_\varrho= \begin{bmatrix}
km+n & m \\
m & n
\end{bmatrix}$.
\end{itemize}
\item[(b)] The following are equivalent.
\begin{itemize}
\item[(i)] ${\lim_{n\to\infty} {r(n)}/{n}=-k}$.
  \vspace{.2cm}
\item[(ii)]  $[1\ \tau_k]^T$ is a right Peron--Frobenius eigenvector for $M_\varrho$.
\item[(iii)] There exist $m \in \NN$ and $n \in \ZZ_{\geqslant 0}$ such that
$
M_\varrho= \begin{bmatrix}
                         n & m \\
                         m &km+n
                       \end{bmatrix}$.
\end{itemize}
\item[(c)] The following are equivalent.
\begin{itemize}
  \item[(i)] ${\lim_{n\to\infty} {r(n)}/{n}=0}$.
  \vspace{.2cm}
  \item[(ii)] $[1\ 1]^T$ is a right Perron--Frobenius eigenvector for $M_\varrho$.
  \item[(iii)] There exist $a,b,c,d \in \ZZ_{\geqslant 0}$ with $a+b=c+d$ such that
$
M_\varrho=\begin{bmatrix}
a & b \\
c & d
\end{bmatrix}$.
\end{itemize}
\end{itemize}
\end{theorem}

\begin{proof} The proof of \textbf{(a)} is analogous to that of Lemma~\ref{lem2a1}, so we omit it.

\textbf{(b)} We have $\lim_{n\to\infty} {r_\ttw(n)}/{n}=-k$ if, and only if, $\lim_{n\to\infty} {r_{\overline{\ttw}}(n)}/{n}=k$. Now, $\overline{\ttw}$ is the one-sided fixed point of the substitution $\overline{\varrho}$. Let ${M}_{\overline{\varrho}}$ be the substitution matrix of $\overline{\varrho}$. Then,
\[
M_\varrho=\begin{bmatrix}
    a & b \\
    c & d
  \end{bmatrix}\quad\mbox{if, and only if,}\quad {M}_{\overline{\varrho}}= \begin{bmatrix}
    d & c \\
    b & a
  \end{bmatrix} \,.
\]
Part (b) now follows now from part (a) applied to the situation of $\overline{\ttw}$.

\textbf{(c)} Let $[u\ 1]$ be the left Perron--Frobenius eigenvector of $M_\varrho$.
Then, by Corollary~\ref{cor:asymp}, $\lim_{n\to\infty} \frac{r(n)}{n}=\frac{u^2-1}{u}$. The equivalence between (i) and (ii) is now clear, and (ii) $\Leftrightarrow$ (iii) is a trivial exercise.
\end{proof}

\section{The Thue--Morse substitution and related words}\label{sec:TM}

We now arrive at our final curiosity of the relative position function. Here, we show that the Thue--Morse word $\ttw_{\rm TM}$ is the only word $\ttw\in\Wa$ on the letters $\tta =1$ and $\ttb =-1$ with the property that the word
\[
r(1)r(2) \cdots r(n) \cdots
\]
is, again, equal to $\ttw$.

To this end, note that The Thue--Morse substitution $\varrho_{\rm TM}$ is defined by
\[
\varrho_{\rm TM}:
\begin{cases}
    \tta \to \tta\ttb \\
    \ttb \to \ttb\tta\, ,
\end{cases}
\]
which has one-sided fixed point $\ttw_{\rm TM}$ and satisfies ${\rm Freq}(\tta) = {\rm Freq}(\ttb) = \frac{1}{2}$ and $\lim_{n\to\infty} {r(n)}/{n} = 0$.

\begin{definition} Let $\ttw = \ell_0 \ell_1 \cdots \ell_n \cdots\in\W$. We define the \emph{dimers}
\[
X_n:=\ell_{2n}\ell_{2n+1} \,.
\]
Note, here, that $\ttw = X_0X_1 \cdots X_n \cdots$, where $X_n \in \{ \tta \tta , \tta \ttb, \ttb \tta , \ttb \ttb \}$.
\end{definition}

In the particular case wherein each dimer contains distinct letters, we can give an explicit formula for $r(n)$ in terms of the dimers $X_n$. Taking into account that the first dimer has index $0$ while $r$ starts at $r(1)$, the following result follows from an easy induction on $n$.

\begin{proposition} Let $\ttw \in \W$ be some word. If all the dimers $X_n$ satisfy $X_n \in \{ \tta \ttb, \ttb \tta \}$ then
\[
r(n)=
\begin{cases}
1 &\mbox{if } X_{n-1} = \tta \ttb \\
-1 &\mbox{if } X_{n-1} = \ttb \tta \,. \qed
\end{cases}
\]
\end{proposition}

In the case of the Thue--Morse word, the dimers are exactly the level-$1$ supertiles $\ttA=\tta \ttb$ and $\ttB=\ttb \tta$ of the substitution $\varrho_{\rm TM}$. This immediately gives the following result.

\begin{theorem} The Thue--Morse word $\ttw_{\rm TM}$ is the only binary word on $\tta=1$ and $\ttb=-1$ starting with $\tta=1$, having the property that ${\ttw}=r(1)r(2)\cdots r(n)\cdots$.
\end{theorem}

\begin{proof}
If $\ttw_{\rm TM}$ is the Thue--Morse word, then by the above,  $r_{\ttw_{\rm TM}}(n+1)=1$ if, and only if, $X_n=\ttA$ if, and only if, $\ell_{n}=\tta = 1$. This shows that Thue--Morse word satisfies this property.

Next, let $\ttw$ be any word on $\tta=1$ and $\ttb=-1$ starting with $\tta=1$ with this property. Since $r_{\ttw}(n) \in \{1,-1\}$ for all $n$, a simple induction shows the dimers satisfy
\[
X_n'=\begin{cases}
\tta \ttb  &\mbox{if } r_{\ttw}(n+1) = 1 \\
\ttb \tta &\mbox{if } r_{\ttw}(n+1) = -1 \,.
\end{cases}
\]
Now, let $\ttw =\ell_0'\ell_1' \cdots \ell_n' \cdots$ and let $\ttw_{\rm TM}=\ell_0\ell_1 \cdots \ell_n \cdots$. We know that $\ell_0=1=\ell'_0$. Now, for each $n \geqslant 0$, we have that $\ell_n = \ell'_n$ implies that $r_{\ttw_{\rm TM}}(n+1)=r_{\ttw}(n+1)$, which implies that $X_n=X_n'$, so that $\ell_{2n} = \ell'_{2n}$ and $\ell_{2n+1} = \ell'_{2n+1}$. A simple induction proves the claim.
\end{proof}

Let us note next that for a word $\ttw \in \W$, the equality's $r_\ttw(1)=1, r_\ttw(2)=-1$ are equivalent to $\ttw \in \tta \ttb \ttb \tta \W$. Therefore, we have
\begin{corollary} The Thue--Morse word $\ttw_{\rm TM}$ is the only binary word in $\tta \ttb \ttb \tta\W$ which is isomorphic to $r(1)r(2)\cdots r(n)\cdots$.
\end{corollary}

In exactly the same way, we can show that the double--double of bits of the Thue--Morse word creates the only word on $\tta =+2$ and $\ttb =-2$ with this property. More generally, we have the following result. Since the proof is identical to the one above, we omit it.

\begin{theorem} Let $k>1$ be a positive integer, let $\ttw_{\rm TM}$ be the Thue--Morse word and let $\ttw = \phi_k\ttw_{\rm TM}$. Then, $\ttw$ is the only binary word on $\tta=k$ and $\ttb=-k$, starting with $\tta=k$, with the property that ${\ttw}=r(1)r(2)\cdots r(n)\cdots$.\qed
\end{theorem}

\subsection*{Acknowledgements}
This work was partially supported by NSERC via grants 2019-05430 (CR) and 2024-04853 (NS), and by a David W. and Helen E. F. Lantis Endowment (MC). This work was initiated during MC's visit to MacEwan University---MC thanks the MacEwan faculty and staff for their hospitality and support.

\bibliographystyle{amsplain}

\begin{thebibliography}{1}

\bibitem{AMP}
J.-P. Allouche, M.~Mend{\`e}s~France, and J.~Peyri{\`e}re, \emph{Automatic {D}irichlet Series}, {J.~Number Theory} \textbf{81} (2000), no.~2, 359--373.

\bibitem{AS1992}
J.-P.~Allouche and J.~Shallit, \emph{The ring of {$k$}-regular
  sequences}, Theoret. Comput. Sci. \textbf{98} (1992), 163--197.

\bibitem{AS1999}
J.-P.~Allouche and J.~Shallit, \emph{The ubiquitous Prouhet--Thue--Morse sequence}, Sequences and their applications (Singapore, 1998), 1--16, Springer Ser. Discrete Math. Theor. Comput. Sci. Springer-Verlag London, Ltd., London, 1999.

\bibitem{ASbook}
J.-P.~Allouche and J.~Shallit, \emph{Automatic Sequences}, Cambridge University Press, 2003.

\bibitem{TAO}
Michael Baake and Uwe Grimm, \emph{Aperiodic order. {V}ol. 1: A Mathematican Invitation}, Encyclopedia of
  Mathematics and its Applications, vol. 149, Cambridge University Press,
  Cambridge, 2013.

\bibitem{B2008}
J.~P.~Bell,
\emph{Logarithmic frequency in morphic sequences},
J.~Th\'eor.~Nombres Bordeaux \textbf{20} (2008), no.~2, 227--241.

\bibitem{EMM}
G.~B.~Escolano, N.~Ma\~nibo, E.~D.~Miro, \textit{Mixing properties and entropy bounds of a family of Pisot random substitutions}, Indagationes Math.~\textbf{33} (2022),no.5, 965--9991.

\bibitem{Que}
M.~ Queff\'{e}lec, \emph{Substitution Dynamical Systems - Spectral Analysis}, Lecture Notes in Math.~1294, Springer, 1987.

 \bibitem{Sing}
B. Sing, \emph{Pisot Substitutions and Beyond}, PhD thesis (Univ. Bielefeld), 2006.
\end{thebibliography}
\providecommand{\bysame}{\leavevmode\hbox to3em{\hrulefill}\thinspace}
\providecommand{\href}[2]{#2}


\end{document}